\documentclass[12pt,english]{amsart}

\usepackage[left=2cm,top=3cm,right=2cm,bottom=3cm]{geometry}

\usepackage{color}
\usepackage[pdfstartpage=1,colorlinks=true,bookmarks=false,pdfstartview={FitH}]{hyperref}

\DeclareMathOperator{\dist}{dist}

\DeclareMathOperator{\diam}{diam}
\DeclareMathOperator{\inrad}{inrad}

\usepackage{amsmath,amssymb}
\usepackage{mathtools}
\usepackage{verbatim}
\usepackage{cleveref}
\DeclarePairedDelimiter\abs{\lvert}{\rvert}
\makeatletter
\let\oldabs\abs
\def\abs{\@ifstar{\oldabs}{\oldabs*}}
\newcommand{\eps}{\varepsilon}

\newcommand{\R}{\mathbb{R}}
\newcommand{\p}{\partial}

\newcommand{\Ds}{(-\Delta)^s}

\newcommand{\PV}{\textnormal{P.V.}\,}
\newcommand{\loc}{\textnormal{loc}}
\newcommand{\norm}[2][]{\left\|{#2}\right\|_{#1}}
\newcommand{\seminorm}[2][]{\left[{#2}\right]_{#1}}

\newcommand{\set}[1]{\left\{#1\right\}}

\newcommand{\bS}{\mathbb{S}}
\newcommand{\cA}{\mathcal{A}}

\newcommand{\cL}{\mathcal{L}}
\newcommand{\cP}{\mathcal{P}}
\newcommand{\cT}{\mathcal{T}}

\newcommand{\textas}{\text{ as }}
\newcommand{\texton}{\text{ on }}

\newcommand{\textor}{\text{ or }}
\newcommand{\Textin}{\text{ in }}
\newcommand{\textfor}{\text{ for }}

\newcommand{\textand}{\text{ and }}

\DeclareMathOperator{\supp}{supp}

\theoremstyle{plain}
\newtheorem{thm}{Theorem}[section]
\newtheorem{lem}[thm]{Lemma}

\newtheorem{cor}[thm]{Corollary}
\newtheorem{prop}[thm]{Proposition}

\theoremstyle{definition}
\newtheorem{defn}[thm]{Definition}

\theoremstyle{remark}
\newtheorem{remark}[thm]{Remark}
\newcommand{\bremark}{\begin{remark} \em}
\newcommand{\eremark}{\end{remark} }

\numberwithin{equation}{section}

\definecolor{g2}{rgb}{0,0.6,0}
\definecolor{r2}{rgb}{0.8,0,0}

\vbadness=\maxdimen

\begin{document}

\title
	{Boundary regularity of an isotropically censored nonlocal operator}

\author{Hardy Chan}
\email{hardy.chan@unibas.ch}
\address{Department of Mathematics and Computer Science, University of Basel, Spiegelgasse 1, 4051 Basel, Switzerland.}


\begin{abstract}
In a bounded domain, we consider a variable range nonlocal operator, which is maximally isotropic in the sense that its radius of interaction equals the distance to the boundary. We establish $C^{1,\alpha}$ boundary regularity and existence results for the Dirichlet problem.
\end{abstract}

\maketitle


\section{Introduction}

\subsection{General setting and the operator}

Let $n\geq1$ and $\Omega\subset\R^n$ be a bounded, connected domain of class $C^{1,1}$. Let $d:\overline\Omega\to[0,+\infty)$ be the distance to the boundary,
\[
d(x)=\dist(x,\p\Omega).
\]
It is convenient to smooth out the distance function by taking $\delta\in C^{1,1}(\overline{\Omega})$ such that $\delta=d$ when $d<d_0$, for a small $d_0>0$.

Let $s\in(0,1)$. Write $B_r(x)=\set{y\in\R^n:|y-x|<r}$ and $B_r=B_r(0)$. We introduce the operator
\[\begin{split}
\cL_\Omega u(x)
&=
    C_{n,s} d(x)^{2s-2}
    \PV\int_{B_{d(x)}(x)}
        \dfrac{u(x)-u(y)}{|x-y|^{n+2s}}
    \,dy\\
&=
    \dfrac{C_{n,s}}{2}
    d(x)^{2s-2}
    \int_{B_{d(x)}}
        \dfrac{2u(x)-u(x+y)-u(x-y)}{|y|^{n+2s}}
    \,dy,
\end{split}\]
where the normalization constant is given by
\begin{equation}\label{eq:Cns}
C_{n,s}^{-1}
=\dfrac{1}{2}r^{2s-2}
    \int_{B_r}
        \dfrac{
            y_n^2
        }{
            |y|^{n+2s}
        }
    \,dy,
    \quad \forall r>0,
    \quad \textor \quad
C_{n,s}=\dfrac{2n(2-2s)}{|\bS^{n-1}|}
    =4(1-s)\dfrac{
        \Gamma(\frac{n+2}{2})
    }{
        \pi^{\frac{n}{2}}
    }.
\end{equation}
This is an isotropic regional fractional Laplacian, and the factor $d(x)^{2s-2}$ is inserted to ensure that $\cL_\Omega u(x)$ converges as $x\to\p\Omega$ to a nontrivial limit, namely $-\Delta u(x)$. See \Cref{lem:lim-op}.

Probabilistically speaking, the operator $\cL_\Omega$ generates a L\'{e}vy type process where a particle at $x\in\Omega$ jumps randomly and isotropically in the largest possible ball $B_{d(x)}(x)$ contained inside $\Omega$.

We list some characteristic properties and consequences.

\begin{itemize}
\item $\cL_\Omega$ enjoys a mid-range maximum principle \Cref{prop:MP-mid}, whose strength lies between the local one for $-\Delta$ and the global one for $(-\Delta)^s$: no absolute minima exist in the region where $\cL_\Omega u\geq 0$, provided that $u$ is non-negative in the domain of interaction outside of that region. As a result, local barriers suffice to control boundary behaviors.

\item The domain of interaction depends on the point of evaluation. Thus, extra effort is needed in the construction of barriers in \Cref{sec:barrier}, even following the established idea \cite{RS5}.

\item $\cL_{\Omega}$ is of order $2s$ but scales quadratically and satisfies the classical Hopf lemma, see \Cref{lem:scaling} and \Cref{lem:BH-2sided}. 

\item $\cL_\Omega$ is not variational. To see this, take $\Omega=(0,1)$. Upon integrating by parts, one immediate sees ``hidden boundary terms'' at $x=1/2$. In higher dimensions the ``hidden boundary'' can be thought of as points having at least two projections to the boundary. Unfortunately, all these points contribute in such a different way that the resulting expression would not be manageable. Consequently, no weak formulations due to integration by parts can be expected. Existence is to be established in the viscosity sense, in \Cref{sec:visc}.

\end{itemize}

Generic and universal constants are denoted by $C,c$. They depend only on $n$, $s$ and $\Omega$.

\subsection{Main results}

Consider the Dirichlet problem
\begin{equation}\label{eq:main}\begin{cases}
\cL_\Omega u=f
    & \Textin \Omega,\\
u=0
    & \texton \p\Omega.
\end{cases}\end{equation}
We study the regularity properties of its classical solutions $u$, meaning that $u\in C^{2s+}(\Omega)\cap C(\overline\Omega)$, where $C^{2s+}(\Omega)=\bigcup_{\beta>0}C^{2s+\beta}(\Omega)$.

Our main results are the following.

\begin{thm}[\emph{A priori} regularity up to boundary]
\label{thm:main}
Suppose 
$\Omega$ is of class $C^{1,1}$, $s\in(0,1)$ and $u\in C^{2s+}(\Omega)\cap C(\overline\Omega)$ solves \eqref{eq:main}. Then there exists $\alpha_0(s)\in(0,1)$ such that for any $\alpha\in(0,\alpha_0(s))$, the following holds.
If either
\begin{enumerate}
\item $s\in(\frac12,1)$, $f\in L^\infty(\Omega)$; or 
\item $s\in(0,\frac12]$, $f\in C^{\alpha+1-2s}(\Omega)$, 
\end{enumerate}
then $u\in C^{1,\alpha}(\overline{\Omega})$, with
\[
\norm[C^{1,\alpha}(\overline{\Omega})]{u}\leq C.
\]
Here $C$ depends on $n,s,\Omega$ and the corresponding norm of $f$.
\end{thm}

Higher regularity up to the boundary remains open. There are two difficulties:
\begin{itemize}
\item The condition $\alpha<1$ is crucially used in the proof of \Cref{prop:Liouville}, where the quadratic growth (linear in $x'$ times linear in $x_n$) contradicts the control $|x|^{1+\alpha}$.
\item A complete study of the action of $\cL_{\R^n_+}$ on monomials is missing (\Cref{lem:mono}).
\end{itemize}

\begin{thm}[Existence]
\label{thm:main-exist}
Suppose $\Omega$ is of class $C^{1,1}$, $s,\beta\in(0,1)$ so that $\beta+2s$ is not an integer, $f\in C^\beta(\Omega)$.  
Then there exists a unique $u\in C^{2s+\beta}(\Omega)\cap C(\overline{\Omega})$ solving \eqref{eq:main}. Moreover, if $\beta>(1-2s)_+$, then $u\in C^{1,\alpha}(\overline{\Omega})$, for some $\alpha\in(0,1)$ given by \Cref{thm:main}.
\end{thm}


\subsection{Main ideas}

Let us explain the heuristics of the proof. By definition, $u$ is $C^{1,\alpha}$ at $0\in\p\Omega$ if
\[
u(x)
=c_0(x\cdot \nu)
    +O(|x|^{1+\alpha}),
        \quad \textas x\to 0.
\]
This is implied (see \Cref{lem:expand}) by the expansion
\begin{equation}\label{eq:expand-2}
\dfrac{u(x)}{d(x)}
=c_0+O(|x|^{\alpha}),
    \quad \textas x\to 0,
\end{equation}
\emph{i.e.} $u/d$ is $C^{\alpha}$ up to $0\in\p\Omega$. 
By building suitable barriers in \Cref{sec:barrier}, we will be able to obtain global H\"{o}lder regularity (\Cref{prop:reg-bdry-global}). This allows the use of a blow-up argument to reduce the problem to a half plane. 

It then suffices to prove a Liouville theorem (\Cref{prop:Liouville}) for solutions to the homogeneous equation in the half space, namely
\[\begin{cases}
\cL_{\R^n_+}u=0
    & \Textin \R^n_+,\\
u=0
    & \texton \p\R^n_+,
\end{cases}\]
under the growth $u(x)=O(|x|^{1+\alpha})$: the only solution is $u=cx_n$.
Since the homogeneous equation is preserved under scaling and tangential differentiations, solutions can be shown to be one dimensional (1D) (e.g. \Cref{lem:Liouville-mild}). Hence, we need to show that solutions $u(x',x_n)$, which is independent of $x'$, to
\[\begin{cases}
\cL_{\R_+^n}u(x_n)=0
    & \text{ for } x_n\in\R_+,\\
u(0)=0,
\end{cases}\]
which grow no faster than $x_n^{1+\alpha}$ must be linear (\Cref{lem:Liouville-1D}). This is in turn implied by the boundary regularity in the half-line (\Cref{prop:BH}), namely
\[
\dfrac{u(x_n)}{x_n}\in C^{\alpha} \quad \text{ up to } 0.
\]
To show this, one simply proves a boundary Harnack inequality (\Cref{lem:BH-quotient}) and an improvement of oscillation (\Cref{lem:BH-osc}).

%

\medskip


\subsection{Related works}

The boundary Harnack inequality for a nonlocal elliptic operator in non-divergence form is proved by Ros-Oton--Serra in \cite[Theorem 1.2]{RS4}. A similar type of nonlocal operator with fixed horizon (range of interaction) at every point has been considered by Bellido and Ortega \cite{BO20}.

\subsection{Generalizations}

We expect that the techniques introduced in this paper should be able to prove similar results in the parabolic setting, and when $\cL_{\Omega}$ is replaced by an analogous integro-differential operator of order $2s$ with any homogeneous kernel.

\section{Preliminary results}

Clearly $\cL_\Omega$ satisfies the global maximum principle.

\begin{lem}[Global maximum principle]
\label{lem:MP-global}
Suppose $u\in C^{2s+}(\Omega)\cap C(\overline\Omega)$ solves
\begin{equation}
\begin{cases}
\cL_\Omega u\geq 0
    & \Textin \Omega,\\
u\geq 0
    & \texton \p\Omega.
\end{cases}
\end{equation}
Then either $u\equiv 0$ in $\overline\Omega$ or $u>0$ in $\Omega$.
\end{lem}

\begin{proof}
At any interior minimum $x_0$, $\cL_\Omega u(x_0)\leq 0$, with strict inequality unless $u\equiv u(x_0)$ in $B_{d(x_0)}(x_0)$. But the latter ball contains a sequence converging to $\p\Omega$ where $u\geq 0$.
\end{proof}

A more careful examination yields the following version of maximum principle. It is especially useful to study the blow-up equation when the domain becomes unbounded.
\begin{prop}[Mid-range strong maximum principle]
\label{prop:MP-mid}
Let ${U}\subset\R^n$ be a domain that is not necessarily bounded. Suppose $G$ is non-empty, bounded, open in $U$. The domain interacting with $G$ is
\[
G_*=\bigcup_{y\in G} B_{d_{U}(y)}(y)
\neq\varnothing.
\]
Suppose $u\in C^{2s+}(G)\cap C(\overline{G_*})$, is a solution to
\[\begin{cases}
\cL_U u \geq 0
    & \Textin G,\\
u\geq 0
    & \Textin \overline{G_*}\setminus G. 
\end{cases}\]
Then $u\geq 0$ in $G$.
\end{prop}

\begin{proof}
If $\min_{\overline{G}}u$ is attained on $\p G \subset \overline{G_*}\setminus G$ and not in $G$, then $u\geq 0$ in $G$ by the Dirichlet condition. If $u(x_0)=\min_{G}u\leq 0$, then
\[\begin{split}
0
&\leq \cL_U u(x_0)\\
&=
    C_{n,s}
    d_U(x_0)^{2s-2}\,
    \PV\int_{B_{d_{U}(x_0)}(x_0)}
        \dfrac{
            u(x_0)-u(y)
        }{
            |x_0-y|^{n+2s}
        }
    \,dy\\
&\leq
    C_{n,s}
    d(x_0)^{2s-2}\,
\left(
    \PV\int_{B_{d_U(x_0)}(x_0) \cap G}
        \dfrac{
            u(x_0)-u(y)
        }{
            |x_0-y|^{n+2s}
        }
    \,dy
    +u(x_0)
    \int_{
        B_{d_{U}(x_0)}(x_0)
        \setminus G
    }
        \dfrac{
            1
        }{
            |x_0-y|^{n+2s}
        }
    \,dy
\right)\\
&\leq 0.
\end{split}\]
Thus equality holds and $u\equiv u(x_0)=0$ in $G$.
\end{proof}

The interior Harnack inequality is known to DiCastro--Kuusi--Palatucci \cite{dCKP}.

\begin{lem}[Interior Harnack inequality]
\label{lem:Harnack-1}
Suppose $B_r(z)\subset \Omega\subset\R^n$, $n\geq1$. If
\[
\cL_{\Omega} u=0
    \quad \Textin B_r(z),
\]
then for any $\eta>0$, there exists $C_0(n,s,\Omega,\eta)>0$ such that
\[
u(x)\leq C_0(\eta) u(y)
	\quad \forall x,y\in B_{(1-\eta)r}(z).
\]
\end{lem}

Now we state the interior estimates which follows from the corresponding result for the restricted fractional Laplacian \cite{RS1,RS3}. We denote
\begin{equation}\label{eq:L1mu}
\norm[L^1_{\mu}(\Omega)]{u}
=\int_{\Omega}
    \dfrac{
        |u(y)|
    }{
        1+|y|^{n+\mu}
    }
\,dy,
    \quad \textfor \mu\in\R.
\end{equation}


\begin{lem}[Interior estimates]
\label{lem:int-est}
Suppose $U\subset\R^n$ is not necessarily bounded, and $B_1\subset B_4\subset U$. Suppose $u\in C^{2s+}(\overline{B_1})\cap C(\overline{U})$ is a solution to
\[
\cL_{U} u=f
    \quad \Textin B_1,
\]
for $f\in L^\infty(B_1)$. Then there exists a constant $C=C(n,s,\epsilon)>0$ such that
\[\begin{cases}
\norm[C^{2s}(\overline{B_{1/2}})]{u}
\leq
    C\left(
        \norm[L^\infty(B_1)]{u}
        +\norm[L^1_{2s}(U)]{u}
        +\norm[L^\infty(B_1)]{d}^{2-2s}
            \norm[L^\infty(B_1)]{f}
    \right)
        & \textfor s\neq \frac12,\\
\norm[C^{1-\epsilon}(\overline{B_{1/2}})]{u}
\leq
    C\left(
        \norm[L^\infty(B_1)]{u}
        +\norm[L^1_{2s}(U)]{u}
        +\norm[L^\infty(B_1)]{d}^{2-2s}
            \norm[L^\infty(B_1)]{f}
    \right)
        & \textfor s=\frac12.
\end{cases}\]
Moreover, if $f\in C^\beta(\overline{B_1})$ for $\beta\in(0,1)$ and $\beta+2s$ is not an integer, then there exists $C=C(n,s,\beta)>0$ such that
\[
\norm[C^{2s+\beta}(\overline{B_{1/2}})]{u}
\leq C\left(
        \norm[C^\beta(\overline{B_1})]{u}
        +\norm[L^\infty(B_{d_1})]{u}
        +\norm[L^1_{2s}(U)]{u}
        +\norm[L^\infty(B_1)]{d}^{2-2s}
            \norm[C^\beta(\overline{B_1})]{f}
    \right),
\]
where $d_1=\norm[L^\infty(B_1)]{d}+2$. 
\end{lem}

\begin{remark}
We notice that, because the operator $\cL_U$ degenerates away from the boundary, so does the estimate (through the term $\norm[L^\infty(B_1)]{d}^{2-2s}$).
\end{remark}

\begin{proof}
Let us extend $u$ by zero outside $U$. At each interior point $x\in U$, one can rewrite the equation in terms of the restricted fractional Laplacian, namely
\[
\Ds u=g[u](x)
    \quad \Textin B_1,
\]
where
\[
g[u](x)
=
\dfrac{
    c_{n,s}
}{
    C_{n,s}
}
\left(
    C_{n,s}\int_{B_{d(x)}^c}
        \dfrac{
            u(x)-u(x+y)
        }{
            |y|^{n+2s}
        }
    \,dy
    +f(x)d(x)^{2-2s}
\right),
\]
with $c_{n,s}=2^{2s}\pi^{-\frac{n}{2}}\Gamma(\frac{n+2s}{2})/|\Gamma(-s)|$ being the normalization constant for $\Ds$.
We first prove the $C^{2s}$ (or $C^{1-\epsilon}$) regularity.
Note that for $x\in B_1$, we have $d(x)\geq 3$ and $B_{d(x)}(x)^c\subset B_1^c$, so
\[\begin{split}
|g[u](x)|
&\leq C\left(
        d(x)^{-2s}|u(x)|
        +\int_{B_1^c}
            \dfrac{|u(y)|}{|y|^{n+2s}}
        \,dy
        +d(x)^{2-2s}|f(x)|
    \right)\\
&\leq
    C\left(
        \norm[L^\infty(B_1)]{u}
        +\norm[L^1_{2s}(U)]{u}
        +\norm[L^\infty(B_1)]{d}^{2-2s}
            \norm[L^\infty(B_1)]{f}
    \right).
\end{split}\]
Suppose first $s\neq\frac12$. By \cite[Theorem 1.1(a)]{RS3},
\[\begin{split}
\norm[C^{2s}(\overline{B_{1/2}})]{u}
&\leq C\left(
        \norm[L^\infty(B_1)]{u}
        +\norm[L^1_{2s}(U)]{u}
        +\norm[L^\infty(B_1)]{g}
    \right)\\
&\leq
    C\left(
        \norm[L^\infty(B_1)]{u}
        +\norm[L^1_{2s}(U)]{u}
        +\norm[L^\infty(B_1)]{d}^{2-2s}
            \norm[L^\infty(B_1)]{f}
    \right).
\end{split}\]
When $s=\frac12$, using again \cite[Theorem 1.1(a)]{RS3}, we replace accordingly the $C^{2s}$ norm by $C^{1-\epsilon}$ norm for any $\epsilon\in(0,1)$, with the constant depending also on $\epsilon$.

Now we prove the higher regularity. Up to multiplicative constants, we decompose $g=g_1+g_2+g_3$
where
\[\begin{split}
g_1(x)&=d(x)^{2-2s}f(x)\\
g_2(x)&=d(x)^{-2s}u(x)\\
g_3(x)&=\int_{U\cap B_{d(x)}(x)^c}
        \dfrac{u(z)}{|z-x|^{n+2s}}\,dz.
\end{split}\]
Since $\seminorm[C^\beta]{\varphi^p}\leq p \norm[L^\infty]{\varphi^{p-1}}\seminorm[C^\beta]{\varphi}$ for $\beta\in(0,1)$ and $p\in\R$, we use the bounds $3\leq d(x)\leq \inrad(U)$ and $\norm[C^{0,1}(\overline{U})]{d}\leq 1$ to control
\[\begin{split}
\seminorm[C^\beta(\overline{B_1})]{g_1}
&\leq C\left(
        \norm[L^\infty(B_1)]{d}^{2-2s}
            \seminorm[C^\beta(\overline{B_1})]{f}
        +\norm[L^\infty(B_1)]{d}^{1-2s}
            \norm[L^\infty(B_1)]{f}
    \right)
\leq C\norm[L^\infty(B_1)]{d}^{2-2s}
    \norm[C^\beta(\overline{B_1})]{f},\\
\seminorm[C^\beta(\overline{B_1})]{g_2}
&\leq C\norm[C^\beta(\overline{B_1})]{u}.
\end{split}\]
For $x,y\in \overline{B_1}$, we express
\[\begin{split}
g_3(x)-g_3(y)
&=\int_{U\cap B_{d(x)}(x)^c}
    u(z)\left(
        \dfrac{1}{|z-x|^{n+2s}}
        -\dfrac{1}{|z-y|^{n+2s}}
    \right)\,dz\\
&\quad\;
+\left(
    \int_{B_{d(x)}(x)^c}
    -\int_{B_{d(y)}(y)^c}
\right)
    \dfrac{u(z)}{|z-y|^{n+2s}}
\,dz.
\end{split}\]
For the integral in the first line, note that $3\leq d(x)\leq |z|$. By mean value theorem, there exist $x_*\in B_1$ such that
\[\begin{split}
\abs{
\int_{U\cap B_{d(x)}(x)^c}
    u(z)\left(
        \dfrac{1}{|z-x|^{n+2s}}
        -\dfrac{1}{|z-y|^{n+2s}}
    \right)\,dz
}
&\leq C
\int_{U\cap B_{d(x)}(x)^c}
    \dfrac{
        |u(z)||x-y|
    }{
        |z-x_*|^{n+2s+1}
    }
\,dz\\
&\leq C\norm[L^1_{2s+1}(U)]{u}
    |x-y|.
\end{split}\]
For the second line we note that there is a nontrivial contribution only in the symmetric difference $B_{d(x)}(x)^c \triangle B_{d(y)}(y)^c$, which lies in an annulus of width at most of order $|x-y|$. More precisely, we have
\[
B_1
\subset
B_{\frac{d(x)+d(y)}{2}-|x-y|}
\subset
B_{d(x)}(x)^c
    \triangle
B_{d(y)}(y)^c
\subset
B_{\frac{d(x)+d(y)}{2}+|x-y|}.
\]
Therefore, using $|z|\geq d(x) \geq 3 \geq 3|y|$ and $\supp u\subset \overline\Omega$,
\[\begin{split}
\abs{
\left(
    \int_{B_{d(x)}(x)^c}
    -\int_{B_{d(y)}(y)^c}
\right)
    \dfrac{u(z)}{|z-y|^{n+2s}}
\,dz
}
&\leq
    \int_{
        \frac{d(x)+d(y)}{2}-|x-y|
        \leq |z|
        \leq \frac{d(x)+d(y)}{2}+|x-y|
    }
    \dfrac{
        |u(z)|
    }{
        |z|^{n+2s}
    }\,dz\\
&\leq
    C\dfrac{
        \norm[L^\infty(B_{(d(x)+d(y))/2+|x-y|})]{u}
    }{
        (\frac{d(x)+d(y)}{2})^{1+2s}
    }
    |x-y|\\
&\leq C\norm[L^\infty(B_{d(x)\vee d(y)+2})]{u}
    |x-y|\\
&\leq C\norm[L^\infty(B_{d_1})]{u}
    |x-y|,
\end{split}\]
where $d_1:=\norm[L^\infty(B_1)]{d}+2$. In summary,
\[\begin{split}
\seminorm[C^\beta(\overline{B_1})]{g}
&\leq C\bigg(
        \norm[C^\beta(\overline{B_1})]{u}
        +\norm[L^\infty(B_{d_1})]{u}
        +\norm[L^1_{2s+1}(U)]{u}\\
&\hspace{2cm}
        +\norm[L^\infty(B_1)]{d}^{2-2s}
            \seminorm[C^\beta(\overline{B_1})]{f}
        +\norm[L^\infty(B_1)]{d}^{1-2s}
            \norm[L^\infty(B_1)]{f}
    \bigg).
\end{split}\]
\[
\norm[C^\beta(\overline{B_1})]{g}
\leq C\bigg(
        \norm[C^\beta(\overline{B_1})]{u}
        +\norm[L^\infty(B_{d_1})]{u}
        +\norm[L^1_{2s+1}(U)]{u}
        +\norm[L^\infty(B_1)]{d}^{2-2s}
            \norm[C^\beta(\overline{B_1})]{f}
    \bigg).
\]
Now, by \cite[Corollary 2.4]{RS1},
\[\begin{split}
\norm[C^{\beta+2s}(\overline{B_{1/2}})]{u}
&\leq C\left(
        \norm[C^\beta(\overline{B_1})]{u}
        +\norm[L^1_{2s}(U)]{u}
        +\norm[C^\beta(\overline{B_1})]{g}
    \right)\\
&\leq C\left(
        \norm[C^\beta(\overline{B_1})]{u}
        +\norm[L^\infty(B_{d_1})]{u}
        +\norm[L^1_{2s}(U)]{u}
        +\norm[L^\infty(B_1)]{d}^{2-2s}
            \norm[C^\beta(\overline{B_1})]{f}
    \right),
\end{split}\]
where we have absorbed $\norm[L^1_{2s+1}(U)]{u}$ by $\norm[L^1_{2s}(U)]{u}$.
\end{proof}

\section{The barriers}
\label{sec:barrier}

\subsection{Super-solution near the boundary}

We construct a barrier in the spirit of \cite{RS3}. Since the domain of interaction varies from point to point, we must compute at all points.

The idea is to consider powers of the distance function to a ball which touches the domain from the outside, since we want the super-solution to be strictly positive except at the contact point, however near the boundary.

\begin{prop}[Super-solution]
\label{prop:barrier}
Let ${U}\subset\R^n$ be a possibly unbounded domain of class $C^{1,1}$. Suppose $x_0\in\p{U}$ can be touched by an exterior ball of radius $b>0$. 
Then, there exists a constant $r_0>0$ and a function $\varphi^{(x_0)}\in C^2 \bigl({U\cap B_{2r_0}(x_0)}\bigr) \cap  C^{0,1}\bigl(\overline{U\cap B_{2r_0}(x_0)}\bigr)$ satisfying
\[\begin{cases}
\cL_{{U}}\varphi^{(x_0)}
\geq 1
    & \Textin {U} \cap B_{r_0}(x_0),\\
\varphi^{(x_0)} \geq 1
    & \Textin {U} \cap \bigl(B_{2r_0}(x_0)\setminus B_{r_0}(x_0)\bigr),\\
\varphi^{(x_0)} \geq 0
    & \Textin \overline{U \cap B_{2r_0}(x_0)},\\
\varphi^{(x_0)} \leq Cd_U
    & \texton U\cap B_{r_0}(x_0) \cap (\R \nu(x_0)),\\
\varphi^{(x_0)}(x_0)=0.
\end{cases}\]
Here $\nu(x_0)$ is the (outer) normal at $x_0\in\p\Omega$, and $r_0$ and $C$ depend only on $n$, $s$ and $b$.
\end{prop}

By a translation we can suppose $x_0=0$. Upon a rotation, the exterior ball condition states that $B_b(-be_n)$ touches $\p{U}$ at $0\in{U}$ from the outside.

The super-solution will be built from
\[
d_{B_b(-be_n)}(y)^p
=(d_{-be_n}(y)-b)^p
=(|y+be_n|-b)^p,
\]
for $p=1$ and $p\sim2^-$. As in \cite{RS3}, at each point $x\in{U}$ we compare $d_{B_b(-be_n)}(y)^p$ with the one-dimensional function
\[
d_{\cT_x}(y)^p
=(d_{\cP_x}(y)-b)^p
=\bigl(
    (y+be_n)\cdot\boldsymbol{v}
    -b
\bigr)^p,
\]
where
\[
\cT_x:=\set{y\in\R^n:(y+be_n)\cdot \boldsymbol{v}=b}
\]
\[
\cP_x:=\set{y\in\R^n:(y+be_n)\cdot \boldsymbol{v}=0}
\]
are the hyperplanes which are orthogonal to
\[
\boldsymbol{v}=\boldsymbol{v}_x:=\dfrac{x+be_n}{|x+be_n|}
\]
and contain respectively $-be_n+b\boldsymbol{v}$ and $-be_n$.

First we consider the planar barrier. This uses the fact that $-\cL_{U}$ behaves like the Laplacian near $\p{U}$. (Notice that this computation is valid for $d_{\cT_x}^p$ only at the point $x$, 
although it is all we will need.)

\begin{lem}\label{lem:super-1D}
For $p>0$ and $x\in U$,
we have
\[
{-\cL}_U\bigl(
    d_{\cT_x}^p
\bigr)(x)
=
    \bigl(
        1+O(|p-2|)
    \bigr)
    d_{\cT_x}(x)^{p-2}.
\]
Here the constant in $O(|p-2|)$ depends only on $n$ and $s$.
\end{lem}

\begin{proof}
Using 
$d_{\cT_x}(x+ty)=d_{\cT_x}(x)+ty\cdot\boldsymbol{v}$ for $t\in[-1,1]$, we have
\[\begin{split}
&\quad\;
    {-\cL}_{U}\bigl(
        d_{\cT_x}^p
    \bigr)(x)\\
&=
    \dfrac{C_{n,s}}{2}
    d_U(x)^{2s-2}
    \int_{|y|<d_U(x)}
        \dfrac{
            \bigl(
                d_{\cT_x}(x)
                +y\cdot\boldsymbol{v}
            \bigr)^p
            +\bigl(
                d_{\cT_x}(x)
                -y\cdot\boldsymbol{v}
            \bigr)^p
            -2
                d_{\cT_x}(x)
            ^p
        }{
            |y|^{n+2s}
        }
    \,dy\\
&=
    \dfrac{C_{n,s}}{2}
    d_U(x)^{2s-2}
        d_{\cT_x}(x)
    ^p
    \int_{|y|<d_U(x)}
        \dfrac{
            \bigl(
                1+\frac{y}{d_{\cT_x}(x)}
                    \cdot\boldsymbol{v}
            \bigr)^p
            +\bigl(
                1-\frac{y}{d_{\cT_x}(x)}
                    \cdot\boldsymbol{v}
            \bigr)^p
            -2
        }{
            |y|^{n+2s}
        }
    \,dy.
\end{split}\]
Changing variable to $y=d_{\cT_x}(x)z$ (recall that $d_{\cT_x}(x)>0$ for $x\neq 0\in\p U$) and choosing another coordinate system for $z$ such that $\boldsymbol{v}$ is the direction of the last coordinate axis, we have (here $\psi$ is defined in \eqref{eq:psi})
\begin{equation}\label{eq:dT-homo}
\begin{split}
    {-\cL}_{U}\bigl(
        d_{\cT_x}^p
    \bigr)(x)
&=
    \dfrac{C_{n,s}}{2}
        d_{\cT_x}(x)
    ^{p-2}
    \left(
        \dfrac{
            d_U(x)
        }{
            d_{\cT_x}(x)
        }
    \right)^{2s-2}
    \int_{
        |z|<\frac{d_U(x)}{d_{\cT_x}(x)}
    }
        \dfrac{
            (1+z_n)^p
            +(1-z_n)^p
            -2
        }{
            |z|^{n+2s}
        }
    \,dz\\
&=
    \psi\left(
        p,\dfrac{d_U(x)}{d_{\cT_x}(x)}
    \right)
    d_{\cT_x}(x)^{p-2}.
\end{split}\end{equation}
By \Cref{lem:psi-p},
\[
{-\cL}_U\bigl(
    d_{\cT_x}^p
\bigr)(x)
=
    \psi\left(
        p,\dfrac{d_U(x)}{d_{\cT_x}(x)}
    \right)
    d_{\cT_x}(x)^{p-2}
=
    \left(
        1+O(|p-2|)
    \right)
    d_{\cT_x}(x)^{p-2},
\]
as desired.
\end{proof}

Next we compare $d_{B_b(-be_n)}^p$ and $d_{\cT_x}^p$ pointwise. For each fixed $x\in U$, hence $\boldsymbol{v}\in\bS^{n-1}$, we denote the projection of a vector $y\in\R^n$ onto $\cP_x$ by
\[
y'=y-(y\cdot\boldsymbol{v})\boldsymbol{v}.
\]

\begin{lem}\label{lem:super-comp-flat-2}
For $x\in U$ and $z\in B_{d_U(x)}(x)$,
\[
0\leq
    \bigl(
        d_{B_b(-be_n)}
        -d_{\cT_x}
    \bigr)(z)
\leq
    \dfrac{
        |(z-x)'|^{2}
    }{
        2b
    }.
\]
\end{lem}

\begin{proof}
For $z\in U$, we express
\[\begin{split}
    \bigl(
        d_{B_b(-be_n)}
        -d_{\cT_x}
    \bigr)(z)
&=
    \bigl(
        d_{B_{-be_n}}(z)-b
    \bigr)
    -\bigl(
        -d_{\cP_x}(z)-b
    \bigr)\\
&=
    \bigl(
        |z+be_n|-b
    \bigr)
    -\bigl(
        (z+be_n)\cdot\boldsymbol{v}-b
    \bigr)\\
&=
	|z+be_n|-(z+be_n)\cdot\boldsymbol{v}\\
&=
    \sqrt{
        \bigl((z+be_n)\cdot\boldsymbol{v}\bigr)^2
        +|(z+be_n)'|^2
    }
    -(z+be_n)\cdot\boldsymbol{v}\\
&=
	\frac{
		|(z+be_n)'|^2
	}{
	    \sqrt{
	       \bigl((z+be_n)\cdot\boldsymbol{v}\bigr)^2
	        +|(z+be_n)'|^2
	    }
		+(z+be_n)\cdot\boldsymbol{v}
	}\\
\end{split}\]
to see it is non-negative. We make the following observations: 
\begin{itemize}
\item Since $(x+be_n)'=0$, $(z+be_n)'=(z-x)'$.
\item The sum of radii of the interior disjoint balls $B_{d_U(x)}(x)$ and $B_b(-be_n)$ is at most the distance between the centers, giving $d_U(x)+b \leq |x+be_n|$. This implies
\[
(z+be_n)\cdot\boldsymbol{v}
=(z-x)\cdot\boldsymbol{v}+|x+be_n|
\geq |x+be_n| - |z-x|
\geq |x+be_n| - d_U(x)
\geq b.
\]
\end{itemize}
Thus
\begin{align*}
\bigl(
	d_{B_b(-be_n)}-d_{\cT_x}
\bigr)(z)
&\leq
	\frac{
		|(z+be_n)'|^2
	}{
		2(z+be_n)\cdot\boldsymbol{v}
	}
\leq
	\frac{
		|(z-x)'|^2
	}{
		2b
	}.
\end{align*}
\end{proof}

Now we can compute $\cL_{U}(d_{-be_n}^p)$ locally near the boundary.



\begin{lem}
\label{lem:super-ball}
Let $p\in[1,2]$ and $x\in U$. Then we have
\begin{equation}
\label{eq:super-ball-1}
{-\cL}_U\bigl(
    d_{B_b(-be_n)}^p
\bigr)(x)
\geq
    \bigl(
        1-C(2-p)
    \bigr)
    d_{B_b(-be_n)}(x)^{p-2}
\end{equation}
and
\begin{equation}
\label{eq:super-ball-2}
\cL_U(d_{B_b(-be_n)})(x)
\geq -Cb^{-1}.
\end{equation}
Here $C$ depends only on $n$ and $s$.
\end{lem}

\begin{proof}
We split
\[\begin{split}
-\cL_{U}\bigl(
    d_{B_b(-be_n)}^p
\bigr)(x)
=-\cL_{U}\bigl(
    d_{B_b(-be_n)}^p-d_{\cT_x}^p
\bigr)(x)
    -\cL_{U}\bigl(
        d_{\cT_x}^p
    \bigr)(x).
\end{split}\]
Since $\cT_x$ is chosen such that $d_{B_b(-be_n)}(x)=d_{\cT_x}(x)$ and $d_{B_b(-be_n)}\geq d_{\cT_x}$ on $B_{d_U(x)}(x)$,
\[\begin{split}
    -\cL_{U}\bigl(
        d_{B_b(-be_n)}^p-d_{\cT_x}^p
    \bigr)(x)
&=
    C_{n,s} d_U(x)^{2s-2}\,
    \PV
    \int_{|y|<d_U(x)}
        \dfrac{
            \bigl(
                d_{B_b(-be_n)}^p-d_{\cT_x}^{p}
            \bigr)(x+y)
        }{
            |y|^{n+2s}
        }
    \,dy
\geq 0.
\end{split}\]
Then \eqref{eq:super-ball-1} follows from \Cref{lem:super-1D}.

For \eqref{eq:super-ball-2}, since $d_{\cT_x}$ is linear hence $\cL_U$-harmonic, by \Cref{lem:super-comp-flat-2} we have
\[\begin{split}
    -\cL_{U}\bigl(
        d_{B_b(-be_n)}
    \bigr)(x)
&=
    -\cL_{U}\bigl(
        d_{B_b(-be_n)}-d_{\cT_x}
    \bigr)(x)\\
&=
    C_{n,s} d_U(x)^{2s-2}\,
    \PV
    \int_{|y|<d_U(x)}
        \dfrac{
            \bigl(
                d_{B_b(-be_n)}-d_{\cT_x}
            \bigr)(x+y)
        }{
            |y|^{n+2s}
        }
    \,dy\\
&\leq
    Cb^{-1}
    d_U(x)^{2s-2}\,
    \PV
    \int_{|y|<d_U(x)}
        \dfrac{
            |y'|^{2}
        }{
            |y|^{n+2s}
        }
    \,dy\\
&\leq Cb^{-1}.
    \qedhere
\end{split}\]

\end{proof}

We are ready to prove \Cref{prop:barrier}.

\begin{proof}[Proof of \Cref{prop:barrier}]
Let 
$\widetilde{\varphi}(x)=2d_{B_b(-be_n)}(x)-d_{B_b(-be_n)}^{p}(x)$, where $p<2$ is chosen (using 
\Cref{lem:super-ball}) such that
\[
{-\cL}_U\bigl(
    d_{B_b(-be_n)}^p
\bigr)(x)
\geq
    \dfrac12 d_{B_b(-be_n)}(x)^{p-2},
        \quad \textfor x\in U.
\]
Then
\[\begin{split}
{\cL}_U\widetilde{\varphi}(x)
&\geq
    -Cb^{-1}
    +\frac12 d_{B_b(-be_n)}(x)^{p-2}
\geq 1,
\end{split}\]
whenever $x$ is close enough to $B_b(-be_n)$, say $d_{B_b(-be_n)}(x)\leq 2r_0<1$. On the other hand, we verify that
\[\begin{split}
\widetilde{\varphi}(x)
=d_{B_b(-be_n)}(x)
    \bigl(
        2-d_{B_b(-be_n)}(x)^{p-1}
    \bigr)
\geq
    d_{B_b(-be_n)}(x),
\end{split}\]
provided that $d_{B_b(-be_n)}(x)\leq 1$. Moreover, on $\set{r_0 \leq d_{B_b(-be_n)} \leq 2r_0}\cap U$ where $2r_0<1$,
\[
\widetilde{\varphi}(x)\geq r_0.
\]
Therefore, $\varphi^{(0)}=r_0^{-1}\widetilde{\varphi}$ is the desired super-solution.
\end{proof}

\subsection{Super-solution for a bounded domain}

A concave paraboloid serves as a simple global super-solution. Choose a coordinate system such that $0\in\Omega$. Let $M=\diam\Omega$ so that $\Omega\subset B_M$. Consider the positive, strictly concave function
\[
\varphi^{(1)}(x)=\dfrac{M^2-|x|^2}{2n}.
\]
When $\Omega=B_M$, this is known as the torsion function.

\begin{lem}\label{lem:super-global}
There holds
\begin{equation*}
\begin{cases}
\cL_\Omega \varphi^{(1)}=1
    & \Textin \Omega,\\
\varphi^{(1)}\geq 0
    & \Textin \overline\Omega.
\end{cases}
\end{equation*}
\end{lem}

\begin{proof}
For any $x\in\Omega$ and $y\in B_{d(x)}$, the parallelogram law implies
\[
2\varphi^{(1)}(x)
    -\varphi^{(1)}(x+y)
    -\varphi^{(1)}(x-y)
=\dfrac{|y|^2}{n}.
\]
Thus 
\[\begin{split}
\cL_\Omega \varphi^{(1)}(x)
&=\dfrac{C_{n,s}}{2}
    d(x)^{2-2s}
    \int_{B_{d(x)}(0)}
        \dfrac{
            2\varphi^{(1)}(x)
            -\varphi^{(1)}(x+y)
            -\varphi^{(1)}(x-y)
        }{
            |y|^{n+2s}
        }
    \,dy\\
&=
    \dfrac{C_{n,s}}{2n}
    d(x)^{2-2s}
    \int_{B_{d(x)}(0)}
    \dfrac{
        |y|^2
    }{
        |y|^{n+2s}
    }\,dy
=1.
    \qedhere
\end{split}\]
\end{proof}

\begin{remark}
By requiring that $\Omega$ be compactly contained in $B_M$, one obtains a strict super-solution. However, we will not need this.
\end{remark}

\section{Boundary Harnack inequality in 1D}

We are interested in one-dimensional Dirichlet problems on the half space $\R^n_+$. Note that for $x=(x',x_n)\in\R^n_+$,
\[\begin{split}
\cL_{\R^n_+}u(x)
&=C_{n,s} x_n^{2s-2}
\PV \int_{|y|<x_n}
	\dfrac{
		u(x)-
		\frac{u(x+y)+u(x-y)}{2}
	}{|y|^{n+2s}}
\,dy\\
&=C_{n,s}|\bS^{n-2}|
\int_0^{x_n}
\int_0^{\pi}
	\dfrac{
		u(x',x_n)
		-\frac{
			u(x'+y',x_n+r\cos\theta)
			+u(x'+y',x_n-r\cos\theta)
		}{2}
	}{
		r^{1+2s}
	}
\sin^{n-2}\theta\,d\theta
\,dr.
\end{split}\]
Throughout this section we assume that $u$ is one-dimensional (i.e. $u$ depends only on $x_n$). Then
\begin{align*}
\cL_{\R^n_+}u(x_n)
=C_{n,s}|\bS^{n-2}|
\int_0^{x_n}
\int_0^{\pi}
	\dfrac{
		u(x_n)
		-\frac{
			u(x_n+r\cos\theta)
			+u(x_n-r\cos\theta)
		}{2}
	}{
		r^{1+2s}
	}
\sin^{n-2}\theta\,d\theta
\,dr,
	\qquad \forall x>0.
\end{align*}
It is convenient to write $x\in\R_+$ in place of $x_n$.
In this section we will prove the following

\begin{prop}[Boundary regularity in 1D]
\label{prop:BH}
Suppose $u\in C^{2s+}((0,1))\cap C([0,2))$ is a 1D solution to
\begin{equation}\label{eq:1D}\begin{cases}
\cL_{\R^n_+}u=0
	& \texton (0,1),\\
u>0
    & \texton (0,2),\\
u(0)=0.
\end{cases}\end{equation}
There exists $\alpha_*\in(0,2s \wedge 1)$ and $C>0$ such that
\[
\norm[C^{\alpha_*}((0,\frac12))]{
    \dfrac{u}{x}
}
\leq Cu(1).
\]
Here the constant $C$ depends on $n$ and $s$.
\end{prop}

\begin{remark}
\label{rmk:model-sol}
Note that the function $x$ is a model solution to \eqref{eq:1D}, and in fact the unique solution on $\R_+$ up to a constant multiple, as we will show in \Cref{lem:Liouville-1D}. In other words, any two solutions are comparable up to the boundary in a H\"{o}lder continuous way.
\end{remark}

\subsection{Preliminaries}

The scaling property \Cref{lem:scaling} allows us to compute the action of $\cL_{\R_+}$ on monomials.

\begin{lem}[Monomials on the half line]
\label{lem:mono}
For any $p\geq 0$, 
\[
\cL_{\R^n_+}x^p
=a(p)x^{p-2}
    \quad \texton \R_+,
\]
where $a(p)=-\psi(p,1)$, as defined in \eqref{eq:psi}.
In particular, $a(0)=a(1)=0$ and $a(2)=-2$.
\end{lem}


\begin{remark}
When $n=1$, by a series expansion,
\begin{align*}
a'(p)
&=
-C_{1,s}
\int_0^1
	\dfrac{
		(1+y)^p \log(1+y)
		+(1-y)^p \log(1-y)
	}{
		y^{1+2s}
	}
\,dy\\
&=\sum_{k\geq 0,\,\ell\geq 1,\,k+\ell\geq 2}
	\binom{p}{k}
	\dfrac{1}{\ell}
	[(-1)^k+(-1)^\ell]
	\int_0^1
		\dfrac{
			y^{k+\ell}
		}{
			y^{1+2s}
		}
	\,dy\\
&=\sum_{m=1}^{\infty}
	\sum_{\ell=1}^{2m}
	\dfrac{
		(-1)^{\ell}
	}{
		\ell(m-s)
	}
	\binom{p}{2m-\ell}.
\end{align*}
However, it is not clear from this expression if $a(p)$ is monotone or signed for large $p$.
\end{remark}

\begin{proof}
Let $r>0$. Applying \Cref{lem:scaling} to $\Omega=r^{-1}\Omega=\R^n_+$ and $u(x)=x^p$, we see that
\[\begin{split}
\cL_{\R^n_+}(rx)^{p}\big|_{x=1}
=r^2 \cL_{\R^n_+}x^{p}\big|_{x=r}.
\end{split}\]
By linearity,
\[\begin{split}
\cL_{\R^n_+}x^p\big|_{x=r}
=r^{p-2} \cL_{\R^n_+}x^{p}\big|_{x=1}.
\end{split}\]
Thus $a(p)=\cL_{\R^n_+}x^p\big|_{x=1}$.
\end{proof}

We will use the following version of strong maximum principle for functions with non-negative data in the adjacent interval of the same length.

\begin{lem}[Strong maximum principle]
\label{lem:MP-1D}
Suppose $u\in C^{2s+}((0,1))\cap C([0,2))$ solves
\begin{equation}\label{eq:MP-1D}\begin{cases}
\cL_{\R^n_+}u\geq 0
    & \Textin (0,1),\\
u\geq 0
    & \Textin [1,2),\\
u(0)\geq 0.
\end{cases}\end{equation}
Then either $u\equiv0$ on $(0,1)$, or
\[
u> 0 \quad \texton (0,1).
\]
\end{lem}

\begin{proof}
This is simply \Cref{prop:MP-mid} with $G=\R^{n-1}\times (0,1)$ and $G_*=\R^{n-1}\times (0,2)$.
\end{proof}

\subsection{Boundary Harnack inequality}

First of all we show \Cref{prop:BH} for $\alpha=0$, using interior Harnack inequality and comparison arguments.

\begin{lem}[Two-sided estimate]
\label{lem:BH-2sided}
Suppose $u\in C^{2s+}((0,1))\cap C([0,2))$ solves
\[\begin{cases}
\cL_{\R^n_+}u=0
    & \Textin (0,1),\\
u>0
    & \Textin (0,2),\\
u(0)=0,
\end{cases}\]
then there exists $C>0$ universal such that
\[
	C^{-1}u(1)
\leq
    \dfrac{u(x)}{x}
\leq C u(1)
    \quad \texton (0,1].
\]
\end{lem}

\begin{proof}
By replacing $u$ by $u/u(1)$ if necessary, we may assume that $u(1)=1$. By \Cref{lem:Harnack-1}, there exists $C>0$ universal such that
\[
C^{-1} \leq u(x) \leq C
    \quad \textfor x\in[\tfrac12,1].
\]
Applying \Cref{lem:MP-1D} to $u-C^{-1}x$ and $2Cx-u$ on $(0,\tfrac12)$ yields the result.
\end{proof}

\begin{cor}[Boundary Harnack inequality]
\label{lem:BH-quotient}
Let $u\in C^{2s+}((0,1))\cap C([0,2))$ be a solution to
\begin{equation}\label{eq:BH-quotient}\begin{cases}
\cL_{\R^n_+}u=0
    & \Textin (0,1),\\
u>0
    & \Textin (0,2),\\
u(0)=0.
\end{cases}\end{equation}
Then there exists $C>0$ universal such that
\[
\sup_{x\in(0,1]}
    \dfrac{u(x)}{x}
\leq C
\inf_{x\in(0,1]}
    \dfrac{u(x)}{x}.
\]
\end{cor}


\subsection{Boundary H\"{o}lder regularity}

\begin{lem}[Improvement of oscillation]
\label{lem:BH-osc}
Suppose $u\in C^{2s+}((0,1))\cap C([0,2))$ solves
\[\begin{cases}
\cL_{\R^n_+}u=0
    & \Textin (0,1),\\
u>0
    & \Textin (0,2),\\
u(0)=0,
\end{cases}\]
For $k=1,2,\dots$, denote
\[
m_k=\inf_{x\in(0,4^{-k})}\frac{u(x)}{x},
    \quad
M_k=\sup_{x\in(0,4^{-k})}\frac{u(x)}{x}.
\]
Then there exists a universal constant $c\in(0,1)$ such that for any $k\geq 1$,
\[
M_{k+1}-m_{k+1}
\leq c (M_k-m_k).
\]
\end{lem}

\begin{proof}
By replacing $u$ by $u/u(1)$ if necessary, we may assume that $u(1)=1$. By \Cref{lem:BH-quotient}, we can take $m_1=C_3^{-1}$ and $M_1=C_3$.
We assume in the following that $u(x)/x$ is not a constant; otherwise we can trivially take $M_k=m_k$ for all $k\geq2$.

Suppose $M_k>m_k>0$ for $k\geq1$ is known, such that
\[
u - m_k x \geq 0
    \quad \textand \quad
M_k x - u \geq 0
    \quad \Textin (0,4^{-k}).
\]
By \Cref{lem:scaling}, both the functions $(u-m_k)(2^{-1}4^{-k}x)$ and $(M_k-u)(2^{-1}4^{-k}x)$ are $\cL_{\R_+}$-harmonic and non-negative on $(0,2)$. As they solve \eqref{eq:MP-1D}, the strong maximum principle \Cref{lem:MP-1D}, they are strictly positive on $(0,1)$. This means that
\[
u - m_k x > 0
    \quad \textand \quad
M_k x - u > 0
    \quad \Textin (0,2^{-1}4^{-k}).
\]
Similarly, the functions $(u-m_k)(4^{-k-1}x)$ and $(M_k-u)(4^{-k-1}x)$ solve \eqref{eq:BH-quotient}, so that \Cref{lem:BH-quotient} implies that
\[\begin{cases}
\displaystyle
\sup_{x\in(0,1)}
    \dfrac{(u-m_k)(4^{-k-1}x)}{x}
\leq C
\inf_{x\in(0,1)}
    \dfrac{(u-m_k)(4^{-k-1}x)}{x},\\
\displaystyle
\sup_{x\in(0,1)}
    \dfrac{(M_k - u)(4^{-k-1}x)}{x}
\leq C
\inf_{x\in(0,1)}
    \dfrac{(M_k - u)(4^{-k-1}x)}{x}.
\end{cases}\]
Rescaling and multiplying throughout by the normalizing factor $4^{k+1}$, we have
\[\begin{cases}
\displaystyle
\sup_{x\in(0,4^{-(k+1)})}
    \dfrac{u(x) - m_k x}{x}
\leq C
\inf_{x\in(0,4^{-(k+1)})}
    \dfrac{u(x) - m_k x}{x},\\
\displaystyle
\sup_{x\in(0,4^{-(k+1)})}
    \dfrac{M_k x - u(x)}{x}
\leq C
\inf_{x\in(0,4^{-(k+1)})}
    \dfrac{M_k x - u(x)}{x}.
\end{cases}\]
This means that
\[\begin{cases}
M_{k+1}-m_k \leq C(m_{k+1}-m_k),\\
M_k-m_{k+1} \leq C(M_k-M_{k+1}).
\end{cases}\]
Adding up these two inequalities,
\[
(M_{k+1}-m_{k+1})+(M_k-m_k)
\leq C\left(
        (M_k-m_k)
        -(M_{k+1}-m_{k+1})
    \right).
\]
Thus
\[
M_{k+1}-m_{k+1}
\leq
    c(M_k-m_k),
        \quad
c = \dfrac{C-1}{C+1}.
\qedhere
\]
\end{proof}

Now a standard iteration yields the H\"{o}lder continuity of the quotient. 

\begin{proof}[Proof of \Cref{prop:BH}]
By replacing $u$ by $u/u(1)$ if necessary, we assume that $u(1)=1$. By \Cref{lem:int-est}, we know that $u\in C^{\beta}$ and $\norm[C^{\beta}(B_{d/2}(d))]{u} \leq Cd^{-\beta}$ for $d>0$. Fix $\theta>(1+\beta)/\beta>1$. Let $x,y\in[0,1/4)$. Write $r=|x-y|$, $d=x\wedge y$. If $r\leq d^{\theta}/2$, then by \Cref{lem:int-est},
\begin{align*}
\abs{\frac{u(x)}{x}-\frac{u(y)}{y}}
&\leq C\frac{1}{x}
	\norm[C^{\beta}(B_{d/2}(d))]{u}
	r^{\beta}
	+Cu(y)
	\norm[C^{\beta}(B_{d/2}(d))]{\frac{1}{x}}
	r^{\beta}\\
&\leq Cx^{-1}d^{-\beta}r^{\beta}
	+Cyd^{-1-\beta}r^{\beta}
\leq Cd^{-1-\beta}r^{\beta}
\leq Cr^{\beta-\frac{1+\beta}{\theta}}.
\end{align*}
If $r\geq \frac{d^\theta}{2}$, then $x,y\in (0,d+r)$ and by iterating \Cref{lem:BH-osc}, we have
\begin{align*}
\abs{\frac{u(x)}{x}-\frac{u(y)}{y}}
\leq \sup_{(0,d+r)}\frac{u}{x}
	-\inf_{(0,d+r)}\frac{u}{x}
\leq C(d+r)^{\beta}
\leq Cr^{\frac{\beta}{\theta}}.
\end{align*}
Hence,
\[
\norm[C^{\alpha}([0,\frac14))]{\frac{u}{x}}
\leq C,
	\quad \text{ for }
\alpha=(\beta-\tfrac{1+\beta}{\theta}) \wedge \tfrac{\beta}{\theta},
\]
as desired.
%
%
\end{proof}

\section{H\"{o}lder regularity up to boundary}

\subsection{Pointwise boundary Harnack inequality}

Using the global maximum principle, we obtain a direct pointwise bound which is good for controlling the interior behavior.

\begin{lem}[Interior control]
\label{lem:MP-for-int}
Let $u\in C^{2s+}(\Omega)\cap C(\overline\Omega)$ be a solution to
\[\begin{cases}
\cL_\Omega u=f
    & \Textin \Omega\\
u=g
    & \texton \p\Omega.
\end{cases}\]
Then
\[
\norm[L^\infty(\Omega)]{u}
\leq \dfrac{(\diam\Omega)^2}{2n}
    \norm[L^\infty(\Omega)]{f}
    +\norm[L^\infty(\p\Omega)]{g}.
\]
\end{lem}

\begin{proof}
Use \Cref{lem:MP-global} on $    \norm[L^\infty(\Omega)]{f}\varphi^{(1)}
    +\norm[L^\infty(\p\Omega)]{g}
    \pm u
$, with $\varphi^{(1)}$ given in \Cref{lem:super-global}.
%
%
\end{proof}


We can now control a solution by the distance function.
\begin{lem}[Global boundary Harnack principle]
\label{lem:BH-f-global}
Suppose $u\in C^{2s+}(\Omega)\cap C(\overline\Omega)$ solves
\[\begin{cases}
\cL_\Omega u = f
    & \Textin \Omega,\\
u=0
    & \texton \p\Omega.
\end{cases}\]
Then there exists a universal constant $C$ such that
\[
\norm[L^\infty(\Omega)]{
    \dfrac{u}{d}
}
\leq C\norm[L^\infty(\Omega)]{f}.
\]
\end{lem}

\begin{proof}
Since $\Omega$ is a bounded domain of class $C^{1,1}$, there exists $b>0$ such that an exterior tangent ball of radius $b$ exists at each point $x_0\in \p\Omega$. By \Cref{prop:barrier}, in a suitable coordinate system there exists $\varphi^{(x_0)}$ 
such that
\[\begin{cases}
\cL_{U}\bigl(
    \norm[L^\infty(\Omega)]{f}\varphi^{(x_0)} \pm u
\bigr) \geq 0
    & \Textin \Omega\cap B_{r_0}(x_0),\\
\norm[L^\infty(\Omega)]{f}\varphi^{(x_0)} \pm u
\geq 0
    & \Textin \bigl(\Omega\cap B_{2r_0}(x_0)\setminus B_{r_0}(x_0)\bigr) \cup \p\Omega.
\end{cases}\]
Since
\[
\bigcup_{x\in U \cap B_{r_0}(x_0)}
    B_{d_\Omega(x)}(x)
\subset
    B_{2r_0}(x_0),
\]
\Cref{prop:MP-mid} applies, we have
\[
|u(x)|\leq \norm[L^\infty(\Omega)]{f}\varphi^{(x_0)}(x).
    \quad \forall x\in U \cap B_{r_0}(x_0)
\]
Since $\varphi^{(x_0)}$ grows linearly away from the boundary, we have
\[
|u(x)|
\leq C\norm[L^\infty(\Omega)]{f}d_\Omega(x),
    \quad \textfor d_\Omega(x)<r_0.
\]
The interior estimate simply follows from \Cref{lem:MP-for-int}.
\end{proof}


We present a local analogue in a half ball $B_r^+=B_r \cap \set{x_n>0}$, where $r>0$.

\begin{lem}[Local boundary Harnack principle]
\label{lem:BH-f-local}
Suppose $u\in C^{2s+}(B_1^+) \cap C(\overline{B_2^+})$ solves
\[\begin{cases}
\cL_{\R^n_+} u = 0
    & \Textin B_1^+,\\
u=0
    & \texton \p\R^n_+ \cap B_2^+.
\end{cases}\]
Then
\[
\norm[L^\infty(B_{1/2}^+)]{
    \dfrac{u}{x_n}
}
\leq
    C\norm[L^\infty(B_2^+)]{u}.
\]
Here $C$ depends only on $n$ and $s$.
\end{lem}

\begin{proof}
Let $x_0\in \p\R^n_+\cap \partial B_1^+$. By \Cref{prop:barrier} with $b=1$, there is a universal $r_0\in(0,1/2)$ such that
\[\begin{cases}
\cL_{\R^n_+}\bigl(
    \norm[L^\infty(B_2^+)]{u} \varphi^{(x_0)} \pm u
\bigr) \geq 0,
    & \Textin B_{r_0}^+(x_0),\\
\norm[L^\infty(B_2^+)]{u} \varphi^{(x_0)} \pm u
\geq 0
    & \Textin \bigl(B_{2r_0}^+(x_0) \setminus B_{r_0}^+(x_0)\bigr) \cup \bigl( \p\R^n_+ \cap B_{r_0}(x_0) \bigr).
\end{cases}\]
By \Cref{prop:MP-mid},
\[
|u|\leq \norm[L^\infty(B_2^+)]{u} \varphi^{(x_0)}
    \quad \Textin B_{r_0}^+(x_0).
\]
Now for each $x\in B_1^+ \cap \set{0<x_n<r_0}$ we choose $x_0=(x',0)$ to obtain
\[
|u(x)|\leq C\norm[L^\infty(B_2^+)]{u}x_n,
    \quad \Textin \set{|x'|<1/2} \times\set{0<x_n<r_0}.
\]
for $C$ universal. The result follows by combining it with the trivial estimate
\[
|u(x)|\leq r_0^{-1}\norm[L^\infty(B_2^+)]{u}x_n
    \quad \Textin \set{|x'|<1/2} \times \set{r_0<x_n<1/2}. \qedhere
\]
%
%
\end{proof}

\normalcolor

\subsection{H\"{o}lder regularity up to boundary}

As above, we give a global and a local result. 
While practically having an order of $2s$ in the interior, the operator satisfies the classical Hopf boundary lemma. Thus the minimum of the two yields the combined regularity. This effect is analogously seen with the spectral fractional Laplacian.

\begin{prop}[Global boundary regularity]
\label{prop:reg-bdry-global}
Suppose 
$u\in C^{2s+}(\Omega)\cap C(\overline\Omega)$ solves
\begin{equation*}
\begin{cases}
\cL_\Omega u = f
    & \Textin \Omega,\\ 
u=0
    & \Textin \p\Omega. 
\end{cases}\end{equation*}
Then for any $\epsilon\in(0,1)$, there exists a constant $C=C(n,s,\Omega,\epsilon)>0$ such that
\[\begin{cases}
\norm[C^{0,1}(\overline\Omega)]{u}
\leq C\norm[L^\infty(\Omega)]{f}
    & \textfor s\in(\frac12,1),\\
\norm[C^{1-\eps}(\overline\Omega)]{u}
\leq C\norm[L^\infty(\Omega)]{f}
    & \textfor s=\frac12,\\
\norm[C^{2s}(\overline\Omega)]{u}
\leq C\norm[L^\infty(\Omega)]{f}
    & \textfor s\in(0,\frac12).\\
\end{cases}\]
\end{prop}


\begin{proof}
By dividing by $\norm[L^\infty(\Omega)]{f}$ if necessary, we can assume
$
\norm[L^\infty(\Omega)]{f}\leq 1.
$
By \Cref{lem:BH-f-global},
\begin{equation}\label{eq:BH-pointwise}
|u|\leq C
    d \quad \Textin \Omega.
\end{equation}
Let
\[
\beta=\begin{cases}
1 & \textfor s\in(\frac12,1),\\
1-\epsilon & \textfor s=\frac12,\\
2s & \textfor s\in(0,\frac12).
\end{cases}\]
We need to show that
\[
|u(x)-u(y)| \leq C
|x-y|^\beta
    \quad \forall x,y\in\overline{\Omega}.
\]
Write $\rho=\min\set{d(x),d(y)}=d(x)$, by interchanging $x$ and $y$ if necessary.
\begin{description}
\item[Case 1] $4|x-y|<\rho$.
Then $y\in B_{\rho/4}(x)\subset B_\rho(x)\subset \Omega$. By \Cref{lem:translation} and \Cref{lem:scaling}, the rescaled function $u_\rho(z)=u(x+\rho z)$ satisfies
\begin{equation}\label{eq:rescaled}
\cL_{\rho^{-1}(\Omega-x)}
    u_\rho(z)
=f_\rho(z)
:=\rho^2 f(x+\rho z)
    \quad \Textin B_{1/4}\subset B_1\subset \Omega.
\end{equation}
Using the interior estimates \Cref{lem:int-est}, we have
\[
\norm[C^{\beta}(B_{1/4})]{u_\rho}
\leq C\left(
        \norm[L^\infty(B_1)]{u_\rho}
        +\norm[L^1_{2s}(\Omega)]{u_\rho}
        +\norm[L^\infty(B_1)]{
        	d_{
            	\rho^{-1}(\Omega-x_0)
        	}
        }^{2-2s}
        \norm[L^\infty(B_1)]{f_\rho}
    \right)
\]
In view of \eqref{eq:BH-pointwise} we observe that
\[\begin{split}
\seminorm[C^{\beta}(B_{1/4})]{u_\rho}
&=
    \rho^\beta \seminorm[C^{\beta}(B_{\rho/4}(x))]{u}\\
\norm[L^\infty(B_1)]{u_\rho}
&\leq C
    \norm[L^\infty(B_\rho(x))]{d}
\leq C
    \rho
\leq C\rho^\beta\\
\end{split}\]
\[\begin{split}
\norm[L^1_{2s}(\Omega)]{u_\rho}
= \norm[L^1_{2s}(\Omega)]{u(x+\rho\cdot)}
&\leq C\norm[L^1_{2s}(\Omega)]{d(x+\rho\cdot)}\\
&\leq
Cd(x)\norm[L^1_{2s}(\Omega)]{1}
+C\rho\norm[L^1_{2s-1}(\Omega)]{1}\\
&\leq C\rho\left(
        1+\int_{
            1\leq |z|
            \leq \rho^{-1} \diam\Omega
        }
            \dfrac{1}{|z|^{n+2s-1}}
        \,dz
    \right)\\
&\leq
\begin{cases}
C\rho
    & \textfor s\in(\frac12,1),\\
C\rho(1+\log\frac{1}{\rho})
    & \textfor s=\frac12,\\
C\rho(1+\rho^{2s-1})
    & \textfor s\in(0,\frac12),\\
\end{cases}\\
&\leq C\rho^\beta.
\end{split}\]
\[\begin{split}
\norm[L^\infty(B_1)]{
        	d_{
            	\rho^{-1}(\Omega-x_0)
        	}
        }^{2-2s}
        \norm[L^\infty(B_1)]{f_\rho}
&\leq
    \norm[L^\infty(\Omega)]{d}^{2-2s}\rho^{2s-2}
    \cdot \rho^2
\leq C
    \rho^{2s}
\leq C
    \rho^\beta.
\end{split}\]
We conclude that
\[
\seminorm[C^{\beta}(B_{\rho/4}(x))]{u}
\leq C
    \quad \textit{ i.e. } \quad
|u(x)-u(y)|\leq
C
|x-y|
    \quad \textfor |x-y|<\frac{\rho}{4}.
\]
\item[Case 2] $|x-y|\geq \frac{\rho}{4}$.
Then
\[\begin{split}
|u(x)-u(y)|
\leq |u(x)|+|u(y)|
&\leq C
    \left(d(x)+d(y)\right)\\
&\leq C
    (2d(x)+|x-y|)\\
&\leq C
    |x-y|
\leq C|x-y|^\beta.
\qedhere
\end{split}\]
\end{description}
\end{proof}


\begin{prop}[Local boundary regularity]
\label{prop:reg-bdry-local}
Suppose 
$u\in C^{2s+\eps}_{\loc}(B_1^+)\cap C(\overline{B_2^+})$ solves
\begin{equation*}
\begin{cases}
\cL_{\R^n_+}u = 0
    & \Textin B_1^+,\\
u=0
    & \Textin \p\R^n_+\cap B_2. 
\end{cases}\end{equation*}
Then for any $\epsilon\in(0,1)$, there exists a constant $C=C(n,s,\epsilon)>0$ such that
\[\begin{cases}
\norm[{C^{0,1}(\overline{B_{1/16}^+})}]{u}
\leq C\norm[L^\infty(B_2^+)]{u}
    & \textfor s\in(\frac12,1),\\
\norm[C^{1-\epsilon}(\overline{B_{1/16}^+})]{u}
\leq C\norm[L^\infty(B_2^+)]{u}
    & \textfor s=\frac12,\\
\norm[C^{2s}(\overline{B_{1/16}^+})]{u}
\leq C\norm[L^\infty(B_2^+)]{u}
    & \textfor s\in(0,\frac12).\\
\end{cases}\]
\end{prop}

\begin{proof}
By normalizing if necessary, we assume $\norm[L^\infty(B_2^+)]{u}\leq 1$.
Let
\[
\beta=\begin{cases}
1 & \textfor s\in(\frac12,1),\\
1-\epsilon & \textfor s=\frac12,\\
2s & \textfor s\in(0,\frac12).
\end{cases}\]
We need to show that
\[
|u(x)-u(y)|
\leq C
|x-y|^\beta
    \quad \forall x,y\in \overline{B_{1/16}^+}. 
\]
Without loss of generality let $\rho=x_n\leq y_n$. By \Cref{lem:BH-f-local} we have
\begin{equation}
\label{eq:int-use-BH-half}
|u|
\leq 
	C
	x_n
    	\quad \Textin \overline{B_{1/2}^+}.
\end{equation}

\medskip
\noindent{\bf Case 1: $4|x-y|<\rho$.}
Then $y\in B_{\rho/4}(x)\subset B_\rho(x)\subset B_1^+$. As in the proof of \Cref{prop:reg-bdry-global}, the rescaled function $u_\rho(z)=u(x+\rho z)$ they satisfy the equation (note $\rho^{-1}\geq 16$)
\[
\cL_{\rho^{-1}(B_2^+-x)}u_\rho=0
    \quad\Textin B_1\subset B_4\subset \rho^{-1}(B_1^+-x).
\]
By \eqref{eq:int-use-BH-half},
\[
|u_\rho(z)|
\leq C
	(x_n+\rho z_n)
	\quad \Textin \rho^{-1}(B_{1/2}^+-x).
\]
%
From \Cref{lem:int-est} we have the estimate
\[
\norm[C^{\beta}(B_{1/4})]{u_\rho}
\leq C\left(
        \norm[L^\infty(B_1)]{u_\rho}
        +\norm[L^1_{2s}(\rho^{-1}(B_1^+-x))]{u_\rho}
    \right).
\]
Therefore, by \eqref{eq:int-use-BH-half},
\[\begin{split}
\rho^{\beta}
    \seminorm[C^{\beta}(B_{\rho/4}(x))]{u}
&\leq C\left(
        \norm[L^\infty(B_\rho(x))]{u}
        +\norm[L^1_{2s}(\rho^{-1}(B_2^+-x))]{x_n+\rho z_n}
    \right)\\
&\leq C\left(
        x_n
        +x_n\norm[L^1_{2s}(\R^{n-1})]{1}
        +\rho\norm[L^1_{2s-1}(\rho^{-1}B_4)]{1}
    \right)\\
&\leq C\rho^\beta.
\end{split}\]

\medskip
\noindent{\bf Case 2: $4|x-y|\geq \rho$.} Then by \eqref{eq:int-use-BH-half},
\[\begin{split}
|u(x)-u(y)|
\leq x_n+y_n
\leq 2\rho+|y_n-x_n|
\leq 9|x-y|
\leq C|x-y|^\beta.
\qedhere
\end{split}\]
\end{proof}

\normalcolor

\section{Liouville-type results}


In this section we classify solutions to homoegenous Dirichlet problems in a half space with controlled growth. Write $\R^n_+=\set{x=(x',x_n)\in\R^{n-1}\times \R_+}$.


\begin{prop}[Liouville-type result]
\label{prop:Liouville}
Let $v$ be a solution to
\begin{equation}\label{eq:v-in-H}\begin{cases}
\cL_{\R^n_+}v=0
    & \Textin \R^n_+,\\
v=0
    & \texton \p\R^n_+,
\end{cases}\end{equation}
which satisfies the growth condition
\begin{equation}\label{eq:v-growth}
|v(x)|\leq C(1+|x|^{1+\alpha}),
\end{equation}
for some $\alpha\in(0,\alpha_*)$ with $\alpha_*\in(0,2s\wedge 1)$ given in \Cref{prop:BH}. Then $v$ is a 1D and linear, i.e.
\[
v(x)=b_0x_n,
\]
for some constant $b_0\in \R$.
\end{prop}

\begin{lem}[Liouville in 1D]
\label{lem:Liouville-1D}
If $\bar{v}$ solves
\[\begin{cases}
\cL_{\R_+}\bar{v}=0
    & \Textin (0,+\infty)\\
\bar{v}(0)=0,
\end{cases}\]
and satisfies the growth condition
\[
|\bar{v}(x)|
\leq C(1+|x|^{1+\alpha}),
\]
where $\alpha\in(0,\alpha_*)$ with $\alpha_*\in(0,2s\wedge 1)$ given in \Cref{prop:BH}. Then
\[
\bar{v}(x)=c_0 x,
\]
for some $c_0\in\R$.
\end{lem}

\begin{proof}
Let
\[
\bar{v}_R(x)=R^{-1-\alpha_*}\bar{v}(R x),
\]
which satisfies the growth condition
\[
|\bar{v}_R(x)|
\leq CR^{-1-\alpha_*}(1+R^{1+\alpha}|x|^{1+\alpha})
\leq CR^{-(\alpha_*-\alpha)}(1+|x|^{1+\alpha}).
\]
In particular,
\[
\norm[L^\infty(0,2)]{\bar{v}_R}
\leq CR^{-(\alpha_*-\alpha)}.
\]
Applying \Cref{prop:BH} to $\bar{v}_R$, we see that
\[
\seminorm[C^{\alpha_*}(0,R)]
	{\dfrac{\bar{v}}{x}}
=\seminorm[C^{\alpha_*}(0,1)]
	{\dfrac{\bar{v}_R}{x}}
\leq C\norm[L^\infty(0,2)]{\bar{v}_R}
\leq CR^{-(\alpha_*-\alpha)}
\to 0,
\]
as $R\to+\infty$. Hence $\bar{v}/x$ is a constant $c_0\in\R$.
\end{proof}

\begin{lem}[Solutions with slow growth vanish]
\label{lem:slow-growth}
Suppose $v$ solves \eqref{eq:v-in-H} with
\[
|v(x)|\leq C(1+|x|^{\beta}),
\]
for $\beta\in[0,\beta_0)$ where
\begin{equation}\label{eq:beta0}
\beta_0=\begin{cases}
1
	& \textfor s\in(1/2,1),\\
1-\eps
	& \textfor s=1/2,\\
2s
	& \textfor s\in(0,1/2).
\end{cases}\end{equation}
Then $v\equiv 0$.
\end{lem}

\begin{proof}
The rescaled function $v_R(x)=R^{-\beta_0}v(Rx)$ satisfies \eqref{eq:v-in-H} and the growth condition
\[
|v_R(x)|
\leq CR^{-\beta_0}(1+R^\beta |x|^\beta)
\leq CR^{-(\beta_0-\beta)}(1+|x|^\beta).
\]
By \Cref{prop:reg-bdry-local},
\[
\seminorm[C^{\beta_0}(B_{R/16}^+)]{v}
=\seminorm[C^{\beta_0}(B_{1/16}^+)]{v_R}
\leq C\norm[L^\infty(B_2^+)]{v_R}
\leq CR^{-(\beta_0-\beta)}
\to 0,
\]
as $R\to\infty$. Hence, $v\equiv v(0)=0$.
\end{proof}

\begin{lem}[Solutions with mild growth are 1D]
\label{lem:Liouville-mild}
Suppose $v$ satisfies \eqref{eq:v-in-H} and the growth condition
\[
|v(x)|\leq C(1+|x|^{\beta})
\]
for $\beta\in[\beta_0,2\beta_0)\cap (0,1+\alpha_*)$ where $\beta_0$ is as in \eqref{eq:beta0} and $\alpha_* \in (0,2s\wedge 1)$ is given in \Cref{prop:BH}. Then $v$ is a 1D, i.e.
\[
v(x)=b_0 x_n,
\]
for some $c_0\in\R$. 
\end{lem}

\begin{proof}
Let $h\in(0,1]$ and $\omega\in \bS^{n-1}\cap \set{x_n=0}$. Write
\[
w(x)
=\dfrac{
	v(x+h\omega)-v(x)
}{
	h^{\beta_0}
},
\]
which satisfies \eqref{eq:v-in-H} and the growth condition (via the rescaling as in \Cref{lem:slow-growth})
\[
\norm[L^\infty(B_{R/32})]{w}
\leq \seminorm[C^{\beta_0}(B_{R/16})]{v}
\leq CR^{-\beta_0}\norm[L^\infty(B_{2R})]{v}
\leq R^{\beta-\beta_0}.
\]
Since $\beta-\beta_0\in[0,\beta_0)$, \Cref{lem:slow-growth} implies $w\equiv0$. Then $v(x+h\omega)=v(x)$ for any $x\in\R^n_+$, $h\in(0,1]$, $\omega\in \bS^{n-1}\cap \set{x_n=0}$. Since $(h\omega)\mathbb{Z}$ is arbitrary on $\set{x_n=0}$, $v$ depends only on $x_n$. By \Cref{lem:Liouville-1D}, $v(x)=b_0x_n$ for some $b_0\in\R$.
\end{proof}

\begin{proof}[Proof of \Cref{prop:Liouville}]
%
%
We will prove by induction in $k$ the following claim: if $v$ satisfies \eqref{eq:v-in-H} and the growth condition
\begin{equation}\label{eq:Liouville-k}
|v(x)|\leq C(1+|x|^{k\beta_0})
\end{equation}
and $k\beta_0<1+\alpha_*$, then $v$ is 1D and linear.

By \Cref{lem:Liouville-mild}, this is true for $k=1$. Suppose the claim is true for $k$ and $v$ is a solution to \eqref{eq:v-in-H} satisfying
\[
|v(x)|\leq C(1+|x|^{(k+1)\beta_0}).
\]
By the rescaling argument and boundary regularity (e.g. in \Cref{lem:Liouville-mild}), the H\"{o}lder difference quotient $\frac{v(x+h\omega)-v(x)}{h^{\beta_0}}$ satisfies \eqref{eq:v-in-H} and \eqref{eq:Liouville-k}. Hence, there exists $b_0(h,\omega)$ such that
\begin{equation}\label{eq:v-periodic}
v(x+h\omega)-v(x)=b_0(h,\omega)x_n.
\end{equation}
By iterating \eqref{eq:v-periodic} for $h=1$, we have
\[
v(x+R\omega)-v(x)=b_0(1,\omega) Rx_n.
\]
for any $R\in \mathbb{N}$. But then for $x=(0,R)$, by \eqref{eq:Liouville-k} (recall that $(k+1)\beta_0<1+\alpha_*<2$) we have
\[
|v(R\omega,R)|
=|v(0,R)+b_0(1,\omega)R^2|
\geq
	|b_0(1,\omega)|R^2
	-CR^{1+\alpha}
\geq \dfrac{|b_0(1,\omega)|}{2}
	R^2,
\]
contradicting \eqref{eq:Liouville-k} unless $b_0(1,\omega)\equiv 0$ for all $\omega\in\bS^{n-1}\cap\set{x_n=0}$. In view of \eqref{eq:v-periodic}, $v$ depends only on $x_n$ and the result follows from \Cref{lem:Liouville-1D}.
%
\end{proof}

%
%
%
%
%

\normalcolor

\section{Proof of the higher regularity}

\begin{proof}[Proof of \Cref{thm:main}]
In view of the interior estimates in \Cref{lem:int-est}, we just need to prove the following expansion: for any $z\in \p\Omega$, there exists $Q_z\in\R$, $r>0$ such that for any $x\in \Omega \cap B_r(z)$,
\begin{equation}\label{eq:main-expand}
|u(x)-Q_z d(x)|\leq C|x-z|^{1+\alpha}.
\end{equation}

Suppose on the contrary that there exists $z\in\p\Omega$ such that \eqref{eq:main-expand} does not hold for any $Q\in\R$, i.e.
\[
\sup_{r\in(0,1]}
    r^{-1-\alpha}
    \norm[L^\infty(B_r(z))]{u-Qd}
=\infty,
    \quad \forall Q\in\R.
\]
We split the proof by contradiction into a number of steps.

\medskip

\noindent\textit{Step 1: Choosing one $Q$ for each $r$. }

For each $r>0$ small, we choose a $Q(r)$ that minimizers $\norm[L^2(B_r(z))]{u-Qd}$, i.e.
\[
Q(r)=\dfrac{
        \int_{B_r(z)}ud\,dx
    }{
        \int_{B_r(z)}d^2\,dx
    }.
\]
We claim that
\begin{equation}\label{eq:blowup-Qr}
\sup_{r\in(0,1]}
    r^{-1-\alpha}
    \norm[L^\infty(B_r(z))]{u-Q(r)d}
=\infty.
\end{equation}
Suppose on the contrary that \eqref{eq:blowup-Qr} does not hold, i.e there exists a (large) $\bar{C}>0$ such that
\begin{equation*}
    \norm[L^\infty(B_r(z))]{u-Q(r)d}
\leq \bar{C} r^{1+\alpha}
    \quad \forall r\in(0,1].
\end{equation*}
Then, for any $x\in B_r(z)$,
\[
|Q(2r)-Q(r)|d(x)
\leq |u(x)-Q(2r)d(x)|+|u(x)-Q(r)d(x)|
\leq 2\bar{C} r^{1+\alpha}.
\]
Since $\sup_{B_r(z)}d=r$,
\[
|Q(2r)-Q(r)|\leq 2\bar{C} r^\alpha.
\]
Since for any $j\geq i \geq 0$,
\[
\abs{Q(2^{-i}r)-Q(2^{-j}r)}
\leq \bar{C} r^\alpha\sum_{k=i}^{j-1}2^{-k\alpha}
\leq C \bar{C} 2^{-i\alpha}r^\alpha,
\]
the limit $Q_0:=\lim_{r\searrow0}Q(r)$ exists,
and by fixing $i=0$ and letting $j\to\infty$,
\[
|Q_0-Q(r)|\leq C\bar{C} r^{\alpha}.
\]
In particular, putting $r=1$ implies $|Q_0|\leq C(\bar{C}+1)$, since $|Q(1)|\leq C$. Hence, for all $r\in(0,1]$,
\[\begin{split}
\norm[L^\infty(B_r(z))]{u-Q_0d}
&\leq \norm[L^\infty(B_r(z))]{u-Q(r)d}
    +\norm[L^\infty(B_r(z))]{(Q_0-Q(r))d}\\
&\leq C\bar{C} r^{1+\alpha}
    +C\bar{C} r^\alpha \sup_{B_r(z)}d
\leq Cr^{1+\alpha},
\end{split}\]
a contradiction. This proves \eqref{eq:blowup-Qr}.

\medskip

\noindent\textit{Step 2: The blow-up sequence and growth bound. }

Now we define the monotone quantity
\[
\theta(r):=
\max_{\bar{r}\in [r,1]}
    (\bar{r})^{-1-\alpha}
    \norm[L^\infty(B_{\bar{r}}(z))]{u-Q(\bar{r})d}.
\]
From $\lim_{r\searrow0}\theta(r)=\infty$, there is a sequence $r_m\to0$ such that
\[
(r_m)^{-1-\alpha}
\norm[L^\infty(B_{r_m}(z))]{u-Q(r_m)d}
=\theta(r_m)\to\infty.
\]
Define the blow-up sequence $v_m:(r_m)^{-1}(\Omega-z)\to\R$,
\[
v_m(x)
:=\dfrac{
        u(z+r_m x)-Q(r_m)d(z+r_m x)
    }{
        (r_m)^{1+\alpha}\theta(r_m)
    },
\]
which satisfies
\begin{equation}\label{eq:vm-norm1}
\norm[L^\infty(B_1)]{v_m}=1
\end{equation}
and, from the choice of $Q(r_m)$,
\begin{equation}\label{eq:vm-ortho}
\int_{B_1} v_m(x)d(z+r_m x)\,dx=0.
\end{equation}
We claim the following growth control
\begin{equation}\label{eq:vm-growth}
\norm[L^\infty(B_R \cap (r_m)^{-1}(\Omega-z))]{v_m}
\leq CR^{1+\alpha}
    \quad \forall R\geq1.
\end{equation}
Indeed, the arguments in \textit{Step 1} (replacing $\theta$ by $\theta(r)$, and the interval $(0,1]$ by $[r,1]$) shows that
\[
|Q(Rr)-Q(r)|\leq C (Rr)^{\alpha} \theta(r)
    \quad \forall R\geq 1.
\]
Also since $\theta$ is non-increasing,
\[
\theta(R r_m)\leq \theta(r_m).
\]
Then (here we implicitly extend suitable functions by $0$ outside $\Omega$)
\[\begin{split}
\norm[L^\infty(B_R)]{v_m}
&=
    \dfrac{1}{(r_m)^{1+\alpha}\theta(r_m)}
    \norm[L^\infty(B_{R r_m}(z))]
        {u-Q(r_m)d}\\
&\leq
    \dfrac{1}{(r_m)^{1+\alpha}\theta(r_m)}
    \left(
        \norm[L^\infty(B_{R r_m}(z))]
            {u-Q(R r_m)d}
        +\abs{Q(R r_m)-Q(r_m)}(R r_m)
    \right)\\
&\leq
    \dfrac{1}{(r_m)^{1+\alpha}\theta(r_m)}
    (R r_m)^{1+\alpha}\theta(R r_m)
    +\dfrac{C}{(r_m)^{1+\alpha}\theta(r_m)}
    (R r_m)^{\alpha} \theta(r_m)
    \cdot (R r_m)\\
&\leq
    R^{1+\alpha} + CR^{1+\alpha}.
\end{split}\]
This proves \eqref{eq:vm-growth}.

\medskip

\noindent\textit{Step 3: Equation for the blow-up sequence. }

Let $\Omega_m=(r_m)^{-1}(\Omega-z)$, which converges to a halfspace $\set{x\cdot e>0}$ as $m\to+\infty$, for $e=-\nu(z)$, the inward normal at $z\in\p\Omega$. By the properties in \Cref{lem:translation} and \Cref{lem:scaling}, the functions $v_m$ satisfy
\begin{equation}\label{eq:vm-eq}\begin{split}
|\cL_{\Omega_m}v_m(x)|
&=
    \dfrac{
        1
    }{
        (r_m)^{1+\alpha}\theta(r_m)
    }
    \abs{
        \cL_{(r_m)^{-1}(\Omega-z)}u(x)
        -Q(r_m)\cL_{(r_m)^{-1}(\Omega-z)}d(x)
    }\\
&=
    \dfrac{
        (r_m)^2
    }{
        (r_m)^{1+\alpha}\theta(r_m)
    }
    \abs{
        \cL_{\Omega}u(z+r_m x)
        -Q(r_m)\cL_{\Omega}d(z+r_m x)
    }
\to 0,
\end{split}\end{equation}
since $\cL_\Omega u$ and $\cL_\Omega d=\cL_{\Omega}\delta$ are bounded
in view of 
\Cref{lem:Ld<}. Now, by \Cref{prop:reg-bdry-global}, $\norm[C^{2\beta}(\Omega_m)]{v_m}\leq C$ for some $\beta>0$. So Arzel\`{a}--Ascoli Theorem asserts a subsequence of $v_m$ uniformly converging on compact sets in $\set{x\cdot e>0}$ to some function $v\in C^\beta(\set{x\cdot e>0})$, $\beta\in(0,1)$. 
Passing to the limit in \eqref{eq:vm-growth} and \eqref{eq:vm-eq} yields
\[
\norm[L^\infty(B_R \cap \set{x\cdot e>0})]{v}
\leq CR^{1+\alpha}
	\quad \forall R\geq1.
\]
and
\[\begin{cases}
\cL_{\set{x\cdot e>0}}v=0
	& \Textin \set{x\cdot e>0}\\
v=0
	& \texton \set{x\cdot e=0}.
\end{cases}\]

\medskip

\noindent\textit{Step 4: Classification of the limit, and the contradiction. }

By \Cref{prop:Liouville}, $v(x)=c_0(x\cdot e)$ for some constant $c_0\in\R$. Using the fact that
\[
\dfrac{d(z+r_m x)}{r_m}
\to x\cdot e
	\quad \textas m\to+\infty,
\]
we pass to the limit in \eqref{eq:vm-ortho} (upon dividing by $r_m$) to see that
\[
0=\int_{B_1\cap \set{x\cdot e>0}}
	v(x)(x\cdot e)
\,dx
=\int_{B_1\cap \set{x\cdot e>0}}
	c_0(x\cdot e)^2
\,dx.
\]
But this implies $c_0=0$ and hence $v=0$, contradicting \eqref{eq:vm-norm1} in the limit $m\to+\infty$. Therefore \eqref{eq:main-expand} holds and the proof is complete.
\end{proof}

\section{Existence of viscosity solution}
\label{sec:visc}

Consider the Dirichlet problem
\begin{equation}\label{eq:visc}\begin{cases}
\cL_{\Omega}u=f
    & \Textin \Omega,\\
u=g
    & \texton \p\Omega.
\end{cases}\end{equation}
We will establish the existence of a continuous viscosity solution using Perron's method, carefully exploiting the mid-range maximum principle that $\cL_{\Omega}$ satisfies. Throughout the section we assume $f\in C^{\alpha}(\Omega)$ and $g\in C(\p\Omega)$, for some $\alpha>0$.

Our goal is to prove the following.

\begin{prop}\label{prop:exist}
Let $\Omega\subset\R^n$ be a bounded domain of class $C^{1,1}$. Let $s\in(0,1)$ and $f\in C^\alpha$ for some $\alpha>0$. For any $f\in C^\alpha(\Omega)$ and $g\in C(\p\Omega)$, There exists a unique $u\in C(\overline{\Omega})$ satisfying \eqref{eq:visc} in the viscosity sense. Moreover, $u\in C^{2s+\alpha}(\Omega)\cap C(\overline{\Omega})$ is a classical solution to \eqref{eq:visc}.
\end{prop}

The notion of viscosity solutions has been successfully used in nonlocal equations, see for example \cite{CS09, SV14, RS5, KKP17}. For the proof, we extend the clean arguments described in \cite{FR20} using the barrier constructed in \Cref{sec:barrier}.

\begin{defn}[Semi-continuous functions]
We denote by
\[
USC(\overline{\Omega}) \qquad \text{(resp. } LSC(\overline{\Omega}) \text{)}
\]
the space of upper (resp. lower) semi-continuous functions in $\overline{\Omega}$. For $u\in L^\infty(\overline{\Omega})$ we define the USC (resp. LSC) envelope as
\[
u^*(x)=\sup_{x_k\to x} \limsup_{k\to\infty} u(x_k)
    \qquad \text{(resp. }
u_*(x)=\inf_{x_k\to x} \liminf_{k\to\infty} u(x_k)
    \text{)}.
\]
%
\end{defn}

We also need a localized definition based on \Cref{prop:MP-mid}.

\begin{defn}[Viscosity solutions]
\label{def:visc-mid}
Let $G$ be an non-empty, open set in $\Omega$. The domain of interaction of $G$ is
\[
G_*=\bigcup_{y\in G} B_{d_{\Omega}(y)}(y).
\]
(Thus, $\cL_{\Omega}=\cL_{G_*}$ in $G$.) Let $f\in C(\Omega)$. We say that $u\in USC(\overline{G_*})$ (resp. $u\in LSC(\overline{G_*})$) is a (mid-range) viscosity sub-solution of
\begin{equation*}
\cL_{G_*}u=f
	\text{ in } G,
\end{equation*}
if for any $x\in \Omega$, any neighborhood $N_x$ of $x$ in $\Omega$ and any $\varphi\in  C^2(N_x)\cap L^1(\overline{G_*})$ with
\[
u(x)=\varphi(x),
	\quad
u\leq \varphi \text{ (resp. $u\geq\varphi$) }
\text{ in } \overline{G_*},
\]
we have
\[
\cL_{G_*}\varphi(x_0) \leq f(x_0)
\qquad \text{(resp. }
\cL_{G_*}\varphi(x_0) \geq f(x_0)
\text{).}
\]
In particular, when $G=\Omega$, $\overline{G_*}=\overline{\Omega}$. We say that $u\in C(\overline{\Omega})$ is a (global) viscosity solution in $\Omega$ if it is both a sub-solution and a super-solution in $\Omega$.
\end{defn}


\begin{remark}\label{rmk:extend}
Global sub-(resp. super-) solutions are necessarily mid-range sub-(resp. super-) solutions (but not vice versa). This is because the test function $\varphi \in L^1(\overline{\Omega})$ can be extended to keep the sign of $u-\varphi$ without affecting the computation of $\cL_{\Omega}\varphi$ at the contact point.
\end{remark}

\subsection{Comparison principle for viscosity solutions}

We generalize \Cref{prop:MP-mid} to viscosity solutions.
%
%

\begin{lem}[Mid-range maximum principle]
Let $G\subseteq \Omega$ be open. Suppose $u\in LSC(\overline{G_*})$ solves, in the viscosity sense, 
\[\begin{cases}
\cL_{G_*} u \geq 0
    & \text{ in } G,\\
u\geq 0
    & \text{ in } \overline{G_*}\setminus G.
\end{cases}\]
Then $u\geq 0$ in $G$.
\end{lem}

\begin{proof}
If not, $\min_{\overline{\Omega}}u=-\delta$ for some $\delta>0$. Using a translated coordinate system if necessary, we assume that $0\in G$. The convex paraboloid
\[
\tilde{\varphi}(x)
=
    -\dfrac{\delta}{2}
    +\dfrac{\delta}{4(1+\diam(G_*)^2)}|x|^2
\]
takes values in $[-\delta/2,-\delta/4]$ and so (by moving down then up) there exists $c>0$ such that
\[
\varphi(x)=\tilde{\varphi}(x)-c
\]
touches $u$ from below at some $x_0\in G$, i.e. $u(x_0)=\varphi(x_0)$ and $u\geq \varphi$ in $G$. By construction $u\geq 0\geq \tilde\varphi\geq \varphi$ in $\overline{G_*}\setminus G$. On the one hand, by \Cref{def:visc-mid},
\[
\cL_{G_*}\varphi(x_0) \geq 0.
\]
On the other hand,
\begin{align*}
\cL_{G_*}\varphi(x_0)
&=\frac{C_{n,s}}{2}
	\frac{\delta}{4(1+\diam(G_*)^2)}
	d_{U}(x_0)^{2s-2}
	\int_{|y|<d_{U}(x_0)}
		\frac{
			2|x|^2-|x+y|^2-|x-y|^2
		}{
			|y|^{n+2s}
		}
	\,dy\\
&=-\frac{C_{n,s}}{2}
	\frac{\delta}{2(1+\diam(G_*)^2)}
	d_{U}(x_0)^{2s-2}
	\int_{|y|<d_{U}(x_0)}
		\frac{
			|y|^2
		}{
			|y|^{n+2s}
		}
	\,dy\\
&=-	\frac{n\delta}{2(1+\diam(G_*)^2)}
<0,
\end{align*}
a contradiction.
\end{proof}

\begin{cor}[Mid-range comparison principle]
\label{cor:comp-mid}
Let $G\subseteq \Omega$ be open. Suppose $u\in USC(\overline{G_*})$, $v\in LSC(\overline{G_*})$ are respectively super- and sub-solutions to \eqref{eq:visc}, i.e.
\[\begin{cases}
\cL_{G_*}u
\leq f 
\leq \cL_{G_*}v
    & \Textin G,\\
u\leq v
    & \texton \overline{G_*}\setminus G,
\end{cases}\]
in the viscosity sense, then
\[
u\leq v
    \quad \Textin G.
\]
\end{cor}

%
%
%

\subsection{Supremum of sub-solutions}

Define the family of admissible sub-solutions as
\begin{align*}
\cA:=\set{v\in USC(\overline{\Omega}): \cL_{\Omega}v\leq f \text{ in $\Omega$, } v\leq g \text{ on $\partial\Omega$}}.
\end{align*}
The pointwise supremum of all sub-solutions in $\cA$ is defined as
\begin{equation}\label{eq:Perron-sup}
u(x):=\sup_{v\in\cA} v(x).
\end{equation}
We will prove that $u$ is a viscosity solution by showing that $u^*=g$ on $\partial\Omega$ so that $u^*=u$, and then verify that $u_*$ is a super-solution so that, by comparison, $u_*=u$.

\begin{prop}[Perron's method]\label{prop:Perron}
The function $u$ defined in \eqref{eq:Perron-sup} lies in $C(\overline{\Omega})$ and is a viscosity solution to \eqref{eq:visc}.
\end{prop}

\begin{lem}\label{lem:Perron-int}
The USC envelope of $u$ defined by \eqref{eq:Perron-sup} is a sub-solution in the interior, i.e.
\[
\cL_{\Omega}u^* \leq f
	\qquad \text{ in } \Omega.
\]
As a result, $\sup_{\overline{\Omega}}u^*\leq C$, for $C>0$ depending only on $n$, $s$, $\norm[L^\infty(\Omega)]{f}$ and $\Omega$.
\end{lem}

\begin{proof}
The proof is the same as \cite[Lemma 4.15]{FR20}, except that the test function $\phi\in C^2$ is chosen such that $u-\phi$ attains its \emph{global} maximum in $\overline{\Omega}$. 
\end{proof}

\begin{lem}\label{lem:Perron-bdry}
The USC envelope of $u$ defined by \eqref{eq:Perron-sup} satisfies the boundary condition, i.e.
\[
u^*|_{\partial\Omega}=g\in C(\partial\Omega).
\]
\end{lem}

\begin{proof}
The proof is similar to \cite[Proof of Theorem 4.17, Step 1]{FR20}, but a mid-range comparison is to be employed. Indeed, for each $x_0\in\partial\Omega$, let $r_0$ be as in \Cref{prop:barrier} and define the barrier
\begin{align*}
w_{\eps}^{\pm}
:=g(x_0)\pm(\eps+k_\eps\varphi^{(x_0)})
	\qquad \text{ in } \overline{B_{2r_0}(x_0)}.
\end{align*}
where $k_\eps$, depending not only on $\eps$ but also on $g$, $\Omega$ and $\sup_{\overline{\Omega}}{u^*}$, is chosen such that
\[
w_{\eps}^- \leq u^* \leq w_{\eps}^+
	\qquad \text{ in } \overline{B_{2r_0}(x_0)} \setminus B_{r_0}(x_0) \supset \overline{B_{r_0}(x_0)_*}.
\]
By \Cref{prop:barrier},
\[
\cL_{\Omega}w_{\eps}^- = -k_\eps 
\leq \cL_{\Omega} u^* 
\leq k_\eps=\cL_{\Omega}w_{\eps}^+
	\qquad \text{ in } B_{r_0}(x_0).
\]
By \Cref{cor:comp-mid},
\[
w_{\eps}^- \leq u^* \leq w_{\eps}^+
\qquad \text{ in } B_{r_0}(x_0).
\]
In particular, since $\varphi^{(x_0)} \in C(\overline{\Omega})$, there exists $\delta(\eps)>0$ such that
\[
g(x_0)-2\eps \leq u^* \leq g(x_0)+2\eps 
	\qquad \text{ in } B_{\delta(\eps)}(x_0).
\]
This implies $\lim_{x_k\to x_0} u^*(x_k)=g(x_0)$ and hence $u^*|_{\partial\Omega}=g \in C(\partial\Omega)$.
\end{proof}

\begin{lem}\label{lem:Perron-USC}
It holds that $u=u^* \in USC(\overline{\Omega})$.
\end{lem}

\begin{proof}
By definition, $u\leq u^*$. By \Cref{lem:Perron-int} and \Cref{lem:Perron-bdry}, $u^*\in \cA$, so $u^* \leq u$.
\end{proof}

\begin{lem}\label{lem:max-visc}
Let $G$ be an open subset of $\Omega$. Suppose $u\in USC(\overline{\Omega})$ satisfies
\[
\cL_{\Omega}u \leq f
	\qquad \text{ in } \Omega,
\]
in the viscosity sense, and $v\in C^2(G) \cap L^\infty(\overline{\Omega})$ satisfies pointwise
\[
\begin{cases}
\cL_{\Omega}v \leq f
	& \text{ in } G,\\
v \leq u
	& \text{ in } \overline{\Omega}\setminus G.
\end{cases}
\]
Then, the maximum $w=u\vee v$ is also a sub-solution in $\Omega$.
\end{lem}

\begin{proof}
Suppose $x\in N_x\subset\Omega$ and $\phi\in C^2(N_x)\cap L^1(\Omega)$ is such that $w(x)=\phi(x)$ and $w\leq \phi$ in $\overline{\Omega}$. We want to show that $\cL_{\Omega}\phi(x) \leq f(x)$. If $w(x)=u(x)$, then since $u\leq w\leq \phi$, the result follows from the fact that $u$ is a viscosity sub-solution. If $w(x)=v(x) \neq u(x)$, then $x\in G$ and (using $v\leq w\leq \phi$) the pointwise computation also gives $\cL_{\Omega}\phi(x) \leq \cL_{\Omega}v \leq f$, as desired.
\end{proof}

\begin{lem}\label{lem:Perron-LSC-envelope}
The LSC envelope of $u$ defined by \eqref{eq:Perron-sup} is a super-solution in the interior, i.e.
\[
\cL_{\Omega}u_* \geq f
	\qquad \text{ in } \Omega.
\]
\end{lem}

\begin{proof}
If $u_*$ is not a super-solution in $\Omega$, then there exists $x\in N_x\subset\Omega$, $\varphi\in C^2(N_x)\cap L^1(\overline{\Omega})$, such that $u(x)=\varphi(x)$, $u_*\geq \varphi$ in $\overline{\Omega}$, while $\cL_{\Omega}\varphi(x)<f(x)$. By replacing $\varphi$ by $\tilde\varphi=\varphi-\eps|\cdot-x|^2$ if necessary (where $\eps$ depends on $\varphi$ and $f$), we can assume that $u_*>\varphi$ in $\overline{\Omega}\setminus\set{x}$. By the continuity of $\cL_{\Omega}\varphi$ and $f$ at $x$, there exist $\delta,\rho>0$ such that
\[
\varphi+\delta<u_*\leq u
	\quad \text{ in } \overline{\Omega}\setminus B_{\rho}(x),
	\quad \text{ and } \quad
\cL_{\Omega}\varphi < f
	\text{ in } B_{\rho}(x).
\]
Now, define $u_\delta=u\vee (\varphi+\delta)$, which is a sub-solution in $\Omega$ due to \Cref{lem:max-visc}. Now $u_\delta\in\cA$ and so $u_\delta \leq u$. But this means that $\varphi+\delta \leq u$ in all of $\overline{\Omega}$ including $x_0$, a contradiction.

Suppose, on the contrary, that $u_*$ is not a super-solution in $\Omega$. Then there exists $x_0\in\Omega$ and $\varphi\in C(\overline{\Omega}) \cap C^2(B_{d_{\Omega}(x_0)}(x_0))$ such that $u_*(x_0)=\varphi(x_0)$, $u_* \geq \varphi$ in $\overline{\Omega}$, while $\cL_{\Omega}\varphi(x_0)<f(x_0)$. By replacing $\varphi$ by $\tilde{\varphi}=\varphi-\eps|x-x_0|^2$ if necessary, where $\eps$ depends on $\varphi$ and $f$, we can assume that $u_*>\varphi$ in $\overline{\Omega} \setminus \set{x_0}$. By continuity of $\varphi$, $f$ and $\cL_{\Omega}\varphi$ at $x_0$, there exists $\delta,\rho>0$ such that
\[
\varphi+\delta < u_* \leq u
	\text{ in } \overline{\Omega}\setminus B_{\rho}(x_0) \supset \overline{B_\rho(x_0)_*}\setminus B_{\rho}(x_0),
	\quad \text{ and } \quad
\cL_{\Omega}\varphi < f
	\text{ in } B_{\rho}(x_0).
\]
Now, define $u_\delta=u \vee (\varphi+\delta)$, which is a sub-solution in $\Omega$ due to \Cref{lem:max-visc}. Now $u_\delta\in\cA$ and so $u_\delta \leq u$. But this means that $\varphi+\delta \leq u$ in all of $\overline{\Omega}$ including $x_0$, a contradiction.
\end{proof}

\begin{lem}\label{lem:Perron-LSC}
It holds that $u=u_*\in LSC(\overline{\Omega})$.
\end{lem}

\begin{proof}
By definition $u_*\leq u$. In view of \Cref{lem:Perron-bdry} and \Cref{lem:Perron-LSC-envelope}, comparing $u_*$ to $u$ via \Cref{cor:comp-mid} gives $u\leq u_*$.
\end{proof}

\begin{proof}[Proof of \Cref{prop:Perron}]
By \Cref{lem:Perron-bdry}, \Cref{lem:Perron-USC} and \Cref{lem:Perron-LSC}, $u\in C(\overline{\Omega})$ is both a sub- and super-solution, and $u=g$ on $\partial\Omega$.
\end{proof}

\subsection{Regularity}

Since it suffices to obtain qualitative interior regularity, we compare to the (restricted) fractional Laplacian in $\R^n$ as in \Cref{lem:int-est}, and invoke the ccorresponding regualrity results in \cite[Chapter 3]{FR-book-s}.

\begin{lem}\label{lem:int-est-visc}
Let $u \in C(\overline{\Omega})$ be as in \Cref{prop:Perron}, with $f\in C^\alpha(\Omega)$. Then $u\in C^{2s+\alpha}(\Omega)$.
\end{lem}

\begin{proof}
We verify that $u$, when extended continuously to a bounded function with compact support outside $\Omega$, is a viscosity solution to
\begin{equation}\label{eq:extended-1}
\Ds u=F[u](x)
    \quad \Textin \Omega,
\end{equation}
where $u$ is  
and
\begin{equation}\label{eq:extended-2}
F[u](x)
=
\dfrac{
    c_{n,s}
}{
    C_{n,s}
}
\left(
    C_{n,s}\int_{B_{d(x)}^c}
        \dfrac{
            u(x)-u(x+y)
        }{
            |y|^{n+2s}
        }
    \,dy
    +f(x)d(x)^{2-2s}
\right).
\end{equation}
Recalling the definition of viscosity solution in \cite[Chapter 3]{FR-book-s}, suppose $x\in N_x\subset \Omega$ and $\phi\in L^1_{2s}(\R^n) \cap C^2(N_x)$ is such that
\[
u(x)=\phi(x)
	\quad \text{ and } \quad
u\leq \phi
	\text{ in } \R^n.
\]
In particular, $\phi\in L^1(\Omega)$ and $u\leq \phi$ in $\overline{\Omega}$. By \Cref{def:visc-mid}, $\cL_{\Omega}\phi(x)\leq f(x)$. This pointwise inequality rearranges to $\Ds \phi \leq F[\phi](x)$. Hence, $u$ is a viscosity sub-solution to \eqref{eq:extended-1}--\eqref{eq:extended-2}. Similarly, $u$ is also a viscosity super-solution. By bootstrapping the regularity result in \cite[Chapter 3]{FR-book-s} (recall that $F[u]$ is as regular as $u$, as in the proof of \Cref{lem:int-est}), $u\in C^{2s+\alpha}(\Omega)$.
\end{proof}

\begin{proof}[Proof of \Cref{prop:exist}]
By \Cref{prop:Perron}, there exists a viscosity solution $u\in C(\overline{\Omega})$. By \Cref{lem:int-est-visc}, $u\in C^{2s+\alpha}(\Omega)$, so it is also a classical solution.
\end{proof}

\begin{proof}[Proof of \Cref{thm:main-exist}]
It follows immediately from  \Cref{prop:exist} and \Cref{thm:main}.
\end{proof}

\appendix

\section{Basic properties of $\cL_\Omega$}

We show that $\cL_\Omega$ reduces to the classical Laplace operator at the boundary, with the choice of the normalization $C_{n,s}d(x)^{2-2s}$.

\begin{lem}[Limit operator]
\label{lem:lim-op}
If $u\in C^{2,\beta}(\Omega)$ for some $\beta>0$, then
\[
\cL_\Omega u(x)
\to
-\Delta u(x),
    \quad \textas x\to\p\Omega.
\]
\end{lem}

\begin{proof}
We compute
\[\begin{split}
\cL_{\Omega}u(x)
&=-C_{n,s}d(x)^{2s-2}
    \PV\int_{B_{d(x)}}
        \dfrac{
            \frac12D^2u(x)[y,y]
            +O\left(
                [D^2 u]_{C^\beta(\Omega)}
                |y|^{2+\beta}
            \right)
        }{
            |y|^{n+2s}
        }
    \,dy\\
&=-C_{n,s}d(x)^{2s-2}
    |\bS^{n-1}|
    \int_0^{d(x)}
    \dfrac{
        \Delta u(x)\frac{r^2}{2n}
    }{
        r^{n+2s}
    }
    r^{n-1}\,dr
    +O\left(
        [D^2 u]_{C^\beta(\Omega)}
        d(x)^\beta
    \right))\\
&\to
    -C_{n,s}\dfrac{|\bS^{n-1}|}{2n(2-2s)}
    \Delta u(x)
=-\Delta u(x),
\end{split}\]
as $d(x)\to 0$.
\end{proof}

A nice bound is available for $\cL_\Omega$ on $C^{1,1}$ functions. We will need it only for $\delta$, a smooth function that agrees with $d_\Omega$ near the boundary.

\begin{lem}\label{lem:Ld<}
We have
\[
\abs{\cL_\Omega\delta(x)}
\leq
    n\seminorm[C^{0,1}(B_{d(x)}(x))]{
        \nabla\delta
    },
        \quad \forall x\in\Omega.
\]
In particular, $\cL_\Omega\delta$ is universally bounded.
\end{lem}

\begin{proof}
Since $\delta\in C^{1,1}(\overline{\Omega})$, by a Taylor expansion with quadratic error and \eqref{eq:Cns},
\[\begin{split}
|\cL_\Omega\delta(x)|
\leq C_{n,s}
    d(x)^{2s-2}
    \int_{B_{d(x)}}
        \dfrac{
            \frac12
            \seminorm[C^{0,1}(B_{d(x)}(x))]{
                \nabla \delta
            }
            |y|^2
        }{
            |y|^{n+2s}
        }
    \,dy
=n\seminorm[C^{0,1}(B_{d(x)}(x))]{\nabla \delta}.
\qedhere
\end{split}\]
\end{proof}

We collect the effect of translation and scaling on $L_\Omega$, since the operator depends heavily on the domain. When various domains are in consideration, we put the domain as a subscript. 

Let $z\in\R^n$. For $u:\Omega\to\R$, define $u(\cdot;z):\Omega-z\to\R$ by
\[
u(x;z)=u(x+z).
\]

\begin{lem}[Translation]
\label{lem:translation}
Let $u\in C^{2s+}(\Omega)$. For any $z\in\R^n$,
\[
\cL_{\Omega-z}u(x;z)
=\cL_{\Omega}u(x+z).
\]
\end{lem}

\begin{proof}
Since
\[
d_{\Omega-z}(x)=d_{\Omega}(x+z)
    \quad \textfor x\in \Omega-z,
\]
we have
\[\begin{split}
\cL_{\Omega-z}u(x;z)
&=C_{n,s}d_{\Omega-z}(x)^{2s-2}
    \PV\int_{B_{d_{\Omega-z}(x)}}
        \dfrac{
            u(x;z)-u(x+y;z)
        }{
            |y|^{n+2s}
        }
    \,dy\\
&=C_{n,s}d_{\Omega}(x+z)^{2s-2}
    \PV\int_{B_{d_{\Omega}(x+z)}}
        \dfrac{
            u(x+z)-u(x+z+y)
        }{
            |y|^{n+2s}
        }
    \,dy\\
&=\cL_{\Omega}(x+z).
\qedhere
\end{split}\]
\end{proof}

Let $r>0$. For $u:\Omega\to\R$, consider the rescaling $u_r:r^{-1}\Omega\to\R$ given by
\[
u_r(x)=u(rx).
\]
\begin{lem}[Scaling]
\label{lem:scaling}
Let $u\in C^{2s+}(\Omega)$. For any $r>0$,
\[
\cL_{r^{-1}\Omega}u_r(x)
=r^2 \cL_{\Omega}u(rx).
\]
\end{lem}

\begin{proof}
Note that
\[
d_{r^{-1}\Omega}(x)=r^{-1}d_{\Omega}(rx),
    \quad \textfor x\in r^{-1}\Omega.
\]
Therefore
\[\begin{split}
\cL_{r^{-1}\Omega}u_r(x)
&=C_{n,s}d_{r^{-1}\Omega}(x)^{2s-2}
    \PV\int_{B_{d_{r^{-1}\Omega}(x)}}
        \dfrac{u_r(x)-u_r(x+y)}{|y|^{n+2s}}
    \,dy\\
&=C_{n,s}
    r^{2-2s}d_{\Omega}(rx)^{2s-2}
    \PV\int_{B_{r^{-1}d_{\Omega}(rx)}}
        \dfrac{
            u(rx)-u(rx+ry)
        }{
            r^{-n-2s}|ry|^{n+2s}
        }
    r^{-n}\,d(ry)\\
&=r^2 \cL_{\Omega}u(rx).
\qedhere
\end{split}\]
\end{proof}

%

\section{An auxiliary function}

Let
\begin{equation}\label{eq:psi}
\begin{split}
\psi(p,t)
&=
    \dfrac{C_{n,s}}{2}
    t^{2s-2}
    \int_{|z|<t}
        \dfrac{
            (1+z_n)^p+(1-z_n)^p-2
        }{
            |z|^{n+2s}
        }
    \,dz\\
&=
    \dfrac{C_{n,s}}{2}
    \int_{|y|<1}
        \dfrac{
            (1+ty_n)^p+(1-ty_n)^p-2
        }{
            t^2|y|^{n+2s}
        }
    \,dy.
\end{split}
\end{equation}
Note that, by \eqref{eq:Cns}, $\psi(2,t)=2$ for all $t>0$.

\begin{lem}\label{lem:psi-p}
For $p>0$ and $t\in[0,1]$,
\[
0\leq \psi_p(p,t) \leq C.
\]
Consequently,
\[
\abs{
    \psi(p,t)-\psi(2,t)
}
\leq C|p-2|.
\]
Here the constant $C$ depends only on $n,s$ and $p$ and it remains bounded as $p\to 2$.
\end{lem}

\begin{proof}
It suffices to bound
\[
\psi_p(p,t)
=
    \dfrac{C_{n,s}}{2}
    \int_{|y|<1}
        \dfrac{
            (1+ty_n)^p\log(1+ty_n)
            +(1-ty_n)^p\log(1-ty_n)
        }{
            t^2|y|^{n+2s}
        }
    \,dy
\geq 0.
\]
When $t\in(0,1/2)$ or $|y_n|<1/2$, we have $|ty_n|<1/2$ and so Taylor expansion gives
\[\begin{split}
&\quad\;
    (1+ty_n)^p\log(1+ty_n)
    +(1-ty_n)^p\log(1-ty_n)
    \\
&\leq
    (1+Cty_n)(ty_n+Ct^2y_n^2)
    +(1+Cty_n)(-ty_n+Ct^2y_n^2)\\
&\leq
    Ct^2y_n^2.
\end{split}\]
When $t\in[1/2,1]$ and $|y_n|\geq 1/2$, we have also $|y|\geq1/2$ so the integrand is bounded (by the boundedness of the function $x\mapsto x^p\log x$ on $[0,2]$).
so
\[\begin{split}
\psi_p(p,t)
&\leq
    C+C\int_{|y_n|<1/2}
        \dfrac{
            y_n^2
        }{
            |y|^{n+2s}
        }
    \,dy
\leq C.
    \qedhere
\end{split}\]
\end{proof}

\begin{lem}\label{lem:psi-q}
There exists $c>0$ depending only on $n$ and $s$ such that as $q\to 0^+$,
\[
-\psi(q,t) \leq cq+O(q^2),
\]
uniformly in $t\in[0,1]$.
\end{lem}

\begin{proof}
Using the Taylor expansion
\begin{align*}
&\quad\;
	(1+y)^q+(1-y)^q-2\\
&=(\log(1+y)+\log(1-y))q
	+\frac{
		(1+y)^{q_*}(\log(1+y))^2
		+(1-y)^{q_*}(\log(1-y))^2
	}{2}q^2\\
&=\log(1-y^2)\cdot q+O(y^2q^2),
\end{align*}
we have
\begin{align*}
-\psi(q,t)
=
\left (
\frac{C_{n,s}}{2}t^{2s-2}
\int_{|y|<t}
	\frac{
		\log(1-y^2)
	}{
		|y|^{n+2s}
	}
\,dy
\right )q
+O(q^2),
\end{align*}
where the error $O(q^2)$ is bounded independently of $t\in[0,1]$. To see that the coefficient of $q$ remains strictly positive as $t\to 0$, we observe that for $|y|<1/2$, $\log(1-y^2) \leq Cy^2$ and the homogeneity similar to \eqref{eq:Cns}.
\end{proof}

\section{Boundary expansions}

\begin{lem}\label{lem:expand}
Suppose $\Omega$ is $C^{1,1}$, $0\in \p\Omega$ and $u\in C^{1,\gamma}(\Omega)$. Let $\nu(x)$ denotes the inward normal of the parallel surface containing $x$, and assume $\nu(0)=e_n$. Suppose
\begin{equation}\label{eq:expand-2a}
\dfrac{u(x)}{d(x)}=c_0+O(|x|^{\gamma})
    \quad \textas x\to0.
\end{equation}
Then
\begin{equation}\label{eq:expand-1}
u(x)=c_0(x\cdot\nu(x))+O(|x|^{1+\gamma})
    \quad \textas x\to0.
\end{equation}
\end{lem}

\begin{proof}
Represent $\p\Omega$ by a graph $x_n=\Phi(x')$ near $x=(x',x_n)=0$, then $\Phi(0)=|\nabla \Phi(0)|=0$ and
\[
x\cdot \nu(x)
=x_n+x\cdot(\nu(x)-e_n)
=d(x)+O(\Phi(x))+O(|D^2\Phi(x)||x|^2)
=d(x)+O(|x|^2).
\]
Suppose \eqref{eq:expand-2a} holds. Then
\[
u(x)=c_0d(x)+O(|x|^{\gamma}d(x)).
\]
Clearly $d(x)=O(|x|)$ as $x\to 0$.
Thus \eqref{eq:expand-1} holds.
\end{proof}

\section*{Acknowledgements}

\noindent 
The author has received funding from the Swiss National Science Foundation under the Grant PZ00P2\_202012/1. He thanks Xavier Fern\'{a}ndez-Real and Xavier Ros-Oton for simulating discussions.

\section*{Data availability statement}

Data sharing not applicable to this article as no datasets were generated or analysed during the current study.

\end{document}